\documentclass[a4paper, oneside]{amsart}
\usepackage{amsmath,amscd}
\usepackage{amssymb,verbatim}
\usepackage{url}

\DeclareMathOperator{\id}{Id}
\DeclareMathOperator{\Aut}{Aut}

\newtheorem{theorem}{Theorem}[section]

\newtheorem{proposition}[theorem]{Proposition}
\newtheorem{corollary}[theorem]{Corollary}
\newtheorem{Claim}[theorem]{Claim}
\newtheorem{defn}[theorem]{Definition}

\newtheorem{Example}[theorem]{Example}
\newtheorem{remark}[theorem]{Remark}

\numberwithin{equation}{section}

\begin{document}
\title[On geometric properties of holomorphic isometries]{On geometric properties of holomorphic isometries between bounded symmetric domains}
\author{Shan Tai Chan}
\address{State Key Laboratory of Mathematical Sciences\\
Academy of Mathematics and Systems Science\\
Chinese Academy of Sciences\\
Beijing 100190\\
China}
\address{Institute of Mathematics\\
Academy of Mathematics and Systems Science\\
Chinese Academy of Sciences\\
Beijing 100190\\
China}
\email{mastchan@amss.ac.cn}


\begin{abstract}
We study holomorphic isometries between bounded symmetric domains with respect to the Bergman metrics up to a normalizing constant. In particular, we first consider a holomorphic isometry from the complex unit ball into an irreducible bounded symmetric domain with respect to the Bergman metrics. In this direction, we show that images of (nonempty) affine-linear sections of the complex unit ball must be the intersections of the image of the holomorphic isometry with certain affine-linear subspaces. We also construct a surjective holomorphic submersion from a certain subdomain of the target bounded symmetric domain onto the complex unit ball such that the image of the holomorphic isometry lies inside the subdomain and the holomorphic isometry is a global holomorphic section of the holomorphic submersion. This construction could be generalized to any holomorphic isometry between bounded symmetric domains with respect to the \emph{canonical K\"ahler metrics}. Using some classical results for complex-analytic subvarieties of Stein manifolds, we have obtained further geometric results for images of such holomorphic isometries.
\end{abstract}
\maketitle

\section{Introduction}
In \cite{Ca53}, Calabi studied holomorphic isometric immersions of complex manifolds into complex space forms. As a continuation of our previous work in \cite{Ch16, CM17, Ch18}, we focus on the study of holomorphic isometries $F:D\to \Omega$ between bounded symmetric domains $D$ and $\Omega$ with respect to the Bergman metrics up to a normalizing constant.
In \cite{Mok12}, Mok proved that $F(D)\subset \Omega$ is totally geodesic provided that all irreducible factors of $D$ are of rank $\ge 2$.
Note that rank-$1$ bounded symmetric domains are precisely the complex unit balls in complex Euclidean spaces.
Therefore, it is natural to restrict our study to holomorphic isometries from complex unit balls into bounded symmetric domains.
We have already established some geometric results for such holomorphic isometries in \cite{Ch16, CM17, Ch18}.

In this article, we observe further geometric properties of holomorphic isometries from complex unit balls into irreducible bounded symmetric domains. 
One of the observations is that given any holomorphic isometry $f$ from the complex unit ball $\mathbb B^n$ into an irreducible bounded symmetric domain with respect to the Bergman metrics, images of (nonempty) affine-linear sections of $\mathbb B^n$ must be the intersection of $f(\mathbb B^n)$ with certain affine-linear subspaces.
For any holomorphic isometry $f:D\to \Omega$ between bounded symmetric domains $D$ and $\Omega\Subset \mathbb C^N$ with respect to certain canonical K\"ahler metrics, we know that $f$ is a proper holomorphic isometric embedding (cf.\,Chan-Xiao-Yuan \cite{CXY17} and Mok \cite{Mok12}), and we will construct a bounded domain of holomorphy $D_f$ in $\mathbb C^N$ and a holomorphic submersion $\widetilde \pi_f:D_f\to D$ such that $S:=f(D)\subset D_f \subset \Omega$, $\widetilde \pi_f\circ f={\rm Id}_D$ and $T_\Omega|_S=T_S\oplus(\ker d\widetilde\pi_f)|_S$ is a holomorphic splitting, i.e., there is a short exact sequence
\[ 0 \to T_S \to T_\Omega|_S \to (\ker d\widetilde\pi_f)|_S \to 0 \]
that splits holomorphically (see Theorem \ref{thm_CI_general}).
In addition, the tangent sequence
\[ 0 \to T_S \to T_\Omega|_S \to N_{S|\Omega} \to 0 \]
also splits holomorphically over $S$ so that $S=f(D)\subset \Omega$ is a \emph{splitting complex submanifold} in the sense of Mok-Ng \cite[Definition 2.1, p.\,1059]{MN17}.
This consequence also follows from Abate-Bracci-Tovena \cite[Example 1.2, p.\,630]{ABT09} by the Steinness of bounded symmetric domains.
On the other hand, we will also find some geometric properties for such holomorphic isometries by using some classical results for complex-analytic subvarieties of Stein manifolds.

Let $W$ be a complex vector space of complex dimension $N$.
Then, we denote by $\mathrm{Gr}(n,W)$ the complex Grassmannian of complex $n$-dimensional subspaces of $W$, i.e.,
\[ \mathrm{Gr}(n,W)=\left\{V\subset W\mid V \text{ is a complex $n$-dimensional subspace of $W$}\right\}. \]
We also denote by $G(n,m):=\mathrm{Gr}(n,\mathbb C^{n+m})$ for any positive integers $n$ and $m$.
Moreover, we denote by $M(n,m;\mathbb C)$ the set of all $n$-by-$m$ matrices with entries lying in $\mathbb C$.
For any $n$-by-$m$ matrix $M$, we denote by $M^T$ the transpose of $M$.
We sometimes identify $M(n,1;\mathbb C)\cong M(1,n;\mathbb C) \cong \mathbb C^n$ for any integer $n\ge 1$ if it does not cause any confusion.

Denote by $\mathbb B^n=\{(z_1,\ldots,z_n)\in \mathbb C^n: \sum_{j=1}^n |z_j|^2<1\}$ the complex unit $n$-ball for any integer $n\ge 1$, and write $\Delta:=\mathbb B^1$ for the open unit disk on the complex plane $\mathbb C$.
We also write $\mathbb B^n(x_0,r):=\{z\in \mathbb C^n: \lVert z-x_0 \rVert_{\mathbb C^n} <r\}$ for any $x_0\in \mathbb C^n$ and any real number $r>0$, where $\lVert z \rVert_{\mathbb C^n}:=\sqrt{\sum_{j=1}^n |z_j|^2}$ for any $z=(z_1,\ldots,z_n)\in \mathbb C^n$.
For any holomorphic map $f=(f_1,\ldots,f_N):U\to W$ between domains $U\subset \mathbb C^n$ and $W\subset \mathbb C^N$, we write $Jf(\zeta_0)=\begin{pmatrix}
{\partial f_i\over \partial \zeta_j}(\zeta_0)
\end{pmatrix}_{1\le i \le N,1\le j\le n} \in M(N,n;\mathbb C)$ for the complex Jacobian matrix of $f$ at $\zeta_0\in U\subset \mathbb C^n$, where $(\zeta_1,\ldots,\zeta_n)\in \mathbb C^n$.

\section{Preservation of complex affine-linear sections}\label{Sec:PCALS}
Write $K_U(\cdot,\cdot)$ for the Bergman kernel of any bounded domain $U\Subset \mathbb C^n$, and denote by $ds_U^2$ the Bergman metric on $U$ whose corresponding K\"ahler form is given by $\omega_{ds_U^2}:=\sqrt{-1}\partial\overline\partial \log K_U(z,z)$, where $z\in U \subset \mathbb C^n$.
For any irreducible bounded symmetric domain $\Omega\Subset \mathbb C^N$ in its Harish-Chandra realization, we have the canonical K\"ahler-Einstein metric $g_\Omega$ on $\Omega$ so that minimal disks of $\Omega$ are of constant Gaussian curvature $-2$.
The corresponding K\"ahler form of $(\Omega,g_\Omega)$ is given by
\[ \omega_{g_\Omega}:=-\sqrt{-1}\partial\overline\partial\log h_\Omega(z,z), \]
where $h_\Omega(z,\xi)$ is a polynomial in $(z,\overline\xi)$.
For instance, for any positive integer $n$, we have $h_{\mathbb B^n}(z,\xi) = 1-\sum_{j=1}^n z_j\overline{\xi_j}$ for $z=(z_1,\ldots,z_n),\xi=(\xi_1,\ldots,\xi_n)\in \mathbb C^n$.
We refer the readers to \cite{Ch16, CM17, Ch18} for details.

Let $U$ be a bounded symmetric domain. We have the decomposition $U=U_1\times\cdots \times U_k$, where $U_j$, $1\le j\le k$, are irreducible bounded symmetric domains.
Then, a K\"ahler metric $g'_U$ on $U$ is said to be a \emph{canonical K\"ahler metric} on $U$ if and only if $g'_U = \sum_{j=1}^k \lambda_j \mathrm{Pr}_j^* g_{U_j}$ for some positive real constants $\lambda_j$, $1\le j\le k$, where $\mathrm{Pr}_j:U\to U_j$ is the canonical projection onto the $j$-th irreducible factor $U_j$ of $U$ for $1\le j\le k$.

Let $f:(\mathbb B^n,kg_{\mathbb B^n})\to (\Omega,g_\Omega)$ be a holomorphic isometric embedding, where $\Omega\Subset \mathbb C^N$ is an irreducible bounded symmetric domain of rank $\ge 2$ in its Harish-Chandra realization.
Assume that $f({\bf 0})={\bf 0}$.
Then, we have the polarized functional equation
\begin{equation}\label{Eq:PolFunEq}
h_\Omega(f(w),f(\zeta)) = \left(1-\sum_{j=1}^n w_j \overline{\zeta_j}\right)^k
\end{equation}
for any $w,\zeta\in \mathbb B^n$ (cf.\, Chan \cite{Ch16}).
For $z=(z_1,\ldots,z_N),\,\xi=(\xi_1,\ldots,\xi_N)\in \mathbb C^N$, we may write
\begin{equation}\label{Eq:sp_form_poly_h}
h_\Omega(z,\xi)
= 1-\sum_{j=1}^N z_j \overline{\xi_j} + \sum_{l=1}^{N'-N} (-1)^{\deg(G_l)} G_l(z) \overline{G_l(\xi)}
\end{equation}
for some homogeneous polynomial $G_l$ of degree $\deg(G_l)\ge 2$, $1\le l\le N'$, because $\Omega$ is of rank $\ge 2$ so that it is not holomorphically isometric to $\mathbb B^N$ (cf.\,Chan \cite{CM17, Ch18}).
Write $f=(f_1,\ldots,f_N)$.
Differentiating Equation (\ref{Eq:PolFunEq}) with respect to $\overline{\zeta_\mu}$, $1\le \mu\le n$, and evaluating at $\zeta={\bf 0}$, we have
\[ \sum_{j=1}^N f_j(w) \overline{{\partial f_j\over \partial \zeta_\mu}({\bf 0})} = k w_\mu \]
for any $w=(w_1,\ldots,w_n)\in \mathbb B^n$.
In particular, we have
\begin{equation}\label{Eq:LinearEq1}
\overline{Jf({\bf 0})}^T\begin{pmatrix}
f_1(w)\\\vdots \\ f_N(w)
\end{pmatrix}=
\begin{pmatrix}
\overline{{\partial f_1\over \partial \zeta_1}({\bf 0})} & \cdots & \overline{{\partial f_N\over \partial \zeta_1}({\bf 0})}\\
\vdots & \ddots& \vdots \\
\overline{{\partial f_1\over \partial \zeta_n}({\bf 0})} & \cdots & \overline{{\partial f_N\over \partial \zeta_n}({\bf 0})}\\
\end{pmatrix}
\begin{pmatrix}
f_1(w)\\\vdots \\ f_N(w)
\end{pmatrix}
= k \begin{pmatrix} w_1 \\ \vdots \\ w_n \end{pmatrix},
\end{equation}
i.e., $\overline{Jf({\bf 0})}^T$ gives rise to a complex linear map $L_f:\mathbb C^N\to\mathbb C^n$ of rank $n$ such that $L_f(f(w)) = k w$ for any $w\in \mathbb B^n$.
In addition, we have $\overline{Jf({\bf 0})}^TJf({\bf 0}) = k{\bf I}_n$.
Actually, the above facts have been observed in \cite{CM17, Ch18} already, but we would like to investigate more properties of the holomorphic isometry from Equation (\ref{Eq:LinearEq1}).
Now, we look for the structure of the image of the restriction $f|_{\Lambda \cap \mathbb B^n}$ in the target bounded symmetric domain $\Omega$ for all nonempty complex affine-linear sections $\Lambda\cap \mathbb B^n$ of $\mathbb B^n$, and we have
\begin{proposition}[Preservation of complex affine-linear sections]
\label{pro:Structure_Slicing}
Let $f:(\mathbb B^n,kg_{\mathbb B^n})$ $\to$ $(\Omega,g_\Omega)$ be a holomorphic isometric embedding such that $f({\bf 0})={\bf 0}$ and $n\ge 2$, where $\Omega\Subset \mathbb C^N$ is an irreducible bounded symmetric domain in its Harish-Chandra realization and $k$ is an integer satisfying $1\le k \le \mathrm{rank}(\Omega)$.
Then, for $1\le m\le n-1$, $f$ induces an injective map 
$f_m^\sharp: \mathrm{Gr}(m,\mathbb C^n) \to \mathrm{Gr}(N-n+m,\mathbb C^N)$
given by
\[ f_m^\sharp(V):=\ker \big({\bf A} \overline{Jf({\bf 0})}^T\big) \quad \forall\;V=\ker ({\bf A})\in \mathrm{Gr}(m,\mathbb C^n),\]
where ${\bf A}\in M(n-m,n;\mathbb C)$ is a matrix of rank $(n-m)$.
In addition, for any $V\in G(m,n-m)$ we have
\[ f_m^\sharp(V)=df_{\bf 0}(V) \oplus \ker\overline{Jf({\bf 0})}^T, \]
which is a complex $(N-n+m)$-dimensional vector subspace of $\mathbb C^N$.
Moreover, we have
\begin{equation}\label{Eq:MapCVSToCVS1}
\begin{split}
f((V+{\bf v})\cap \mathbb B^n)
=(f_m^\sharp(V) + Jf({\bf 0})  {\bf v})\cap f(\mathbb B^n)
\subseteq
(f_m^\sharp(V) + Jf({\bf 0})  {\bf v})\cap \Omega
\end{split}
\end{equation}
and
\begin{equation}\label{Eq:MapCVSToCVS2}
\begin{split}
f((V+{\bf v})\cap \mathbb B^n)
=(f_m^\sharp(V) + f({\bf v}))\cap f(\mathbb B^n)
\subseteq
(f_m^\sharp(V) + f({\bf v}))\cap \Omega
\end{split}
\end{equation}
for any complex $m$-dimensional vector subspace $V$ of $\mathbb C^n$ and any ${\bf v}\in \mathbb B^n\subset \mathbb C^n$, $1\le m\le n-1$.

In other words, for any complex $m$-dimensional affine-linear subspace $\Lambda_m$ of $\mathbb C^n$ such that $\Lambda_m\cap \mathbb B^n\neq \varnothing$ {\rm(}$1\le m\le n-1${\rm)}, there exists a complex $(N-n+m)$-dimensional affine-linear subspace $\Pi_{N-n+m}$ of $\mathbb C^N$ such that
\[ f(\Lambda_m\cap \mathbb B^n)\subset \Pi_{N-n+m} \cap\Omega. \]
In particular, $f$ takes hyperplanes into hyperplanes in the sense of Dor {\rm\cite{Do91}}.
\end{proposition}
\begin{proof}
Write $f=(f_1,\ldots,f_N)$ and $S:=f(\mathbb B^n)$.
From the assumption, we have
\[ \overline{Jf({\bf 0})}^T \begin{pmatrix}
f_1(w)\\\vdots \\ f_N(w)
\end{pmatrix}
= k \begin{pmatrix} w_1 \\ \vdots \\ w_n \end{pmatrix} \]
for any $(w_1,\ldots,w_n)\in \mathbb B^n$.
We restrict to some complex $m$-dimensional vector subspace $V\subseteq \mathbb C^n$ for $1\le m\le n-1$.
Then, we may write
\[ V=\left\{ (w_1,\ldots,w_n)\in \mathbb C^n: {\bf A} \begin{pmatrix}
w_1,\ldots,w_n
\end{pmatrix}^T = {\bf 0} \right\} =\ker {\bf A} \]
for some matrix ${\bf A}={\bf A}_V\in M(n-m,n;\mathbb C)$ of full rank $n-m$.
For any $w\in \mathbb B^n\cap V$ we have
${\bf A}\;\overline{Jf({\bf 0})}^T\begin{pmatrix}
f_1(w),\ldots,f_N(w)
\end{pmatrix}^T
= {\bf 0}$ so that $f(\mathbb B^n\cap V) \subseteq W \cap S$, where
\[ W:= \left\{ (z_1,\ldots,z_N)\in \mathbb C^N:
{\bf A}\;\overline{Jf({\bf 0})}^T \begin{pmatrix}
z_1,\ldots,z_N
\end{pmatrix}^T={\bf 0}\right\}=\ker\left({\bf A}\;\overline{Jf({\bf 0})}^T\right). \]
Since $\mathrm{rank}({\bf A}\;\overline{Jf({\bf 0})}^T) = \mathrm{rank}({\bf A}) = n-m$, we have
$\dim_{\mathbb C}(W) = N-n+m$.
In general, for any ${\bf v}\in \mathbb B^n$ we have
$f((V+{\bf v})\cap \mathbb B^n)
\subseteq
(W + Jf({\bf 0}){\bf v})\cap S$.
Actually, for any $w\in (V+{\bf v})\cap \mathbb B^n$ we have
\[ {\bf A}\;\overline{Jf({\bf 0})}^T\begin{pmatrix}
f_1(w),\ldots,f_N(w)
\end{pmatrix}^T
= k \,{\bf A} \begin{pmatrix}
w_1,\ldots,w_n
\end{pmatrix}^T = k\, {\bf A}  {\bf v} \]
so that
${\bf A}\;\overline{Jf({\bf 0})}^T  \left( 
\begin{pmatrix}
f_1(w),\ldots,f_N(w)
\end{pmatrix}^T
-Jf({\bf 0}){\bf v}
\right) = {\bf 0}$.
In particular, we have
\begin{equation}\label{Eq:RestrictCVSubS}
f(\mathbb B^n\cap (\ker {\bf A}+{\bf v})) 
\subseteq \left(\ker\big({\bf A}\;\overline{Jf({\bf 0})}^T\big)+ Jf({\bf 0})  {\bf v}\right)\cap S
\end{equation}
for any matrix ${\bf A}\in M(n-m,n;\mathbb C)$ of full rank $n-m$, $1\le m\le n-1$.
Similarly, by Equation (\ref{Eq:LinearEq1}) we have
\[ {\bf A}\;\overline{Jf({\bf 0})}^T  \left( 
\begin{pmatrix}
f_1(w),\ldots,f_N(w)
\end{pmatrix}^T
-f({\bf v})
\right) = k {\bf A} \left( \begin{pmatrix}
w_1,\ldots,w_n
\end{pmatrix}^T - {\bf v}\right) = {\bf 0} \]
because $\begin{pmatrix}
w_1,\ldots,w_n
\end{pmatrix}^T - {\bf v}\in V=\ker {\bf A}$.
Hence, we also have
\begin{equation}\label{Eq:RestrictCVSubS2}
f(\mathbb B^n\cap (\ker {\bf A}+{\bf v})) 
\subseteq \left(\ker\big({\bf A}\;\overline{Jf({\bf 0})}^T\big)+ f({\bf v})\right)\cap S.
\end{equation}

For any $\xi\in \left(\ker\big({\bf A}\,\overline{Jf({\bf 0})}^T\big)+Jf({\bf 0})  {\bf v}\right) \cap S$, we have ${\bf A}\;\overline{Jf({\bf 0})}^T (\xi^T-Jf({\bf 0})  {\bf v}) = {\bf 0}$ and $\xi = f(w)$ for some $w\in \mathbb B^n$.
In other words, $\xi \in f(\mathcal V)$, where
\[ \mathcal V:=
\left\{w\in \mathbb B^n :{\bf A}\;\overline{Jf({\bf 0})}^T \begin{pmatrix}
f_1(w),\ldots,f_N(w)
\end{pmatrix}^T = k \;{\bf A} {\bf v}\right\}.\]
But then we have
$\mathcal V = \left\{w\in \mathbb B^n :{\bf A}\begin{pmatrix}
w_1,\ldots,w_n
\end{pmatrix}^T = {\bf A} {\bf v}\right\} = \mathbb B^n\cap (\ker {\bf A} + {\bf v})$ so that $\left(\ker\big({\bf A}\;\overline{Jf({\bf 0})}^T\big)+Jf({\bf 0})  {\bf v}\right) \cap S \subseteq f(\mathbb B^n\cap (\ker {\bf A} + {\bf v}))$.
This shows that 
\begin{equation}\label{Eq:RestrictCVSubS2}
f(\mathbb B^n\cap (\ker {\bf A}+{\bf v})) = \left(\ker\big({\bf A}\,\overline{Jf({\bf 0})}^T\big)+Jf({\bf 0}) {\bf v}\right) \cap S 
\end{equation}
for any matrix ${\bf A}\in M(n-m,n;\mathbb C)$ of full rank $n-m$, $1\le m\le n-1$, and ${\bf v}\in \mathbb C^n$ such that $\mathbb B^n\cap (\ker {\bf A}+{\bf v})\neq \varnothing$.

Writing $V=\ker {\bf A}$ for some rank-$(n-m)$ matrix ${\bf A}\in M(n-m,n;\mathbb C)$, we define 
\begin{equation}
f_m^\sharp(V):=\ker \big({\bf A} \overline{Jf({\bf 0})}^T\big).
\end{equation}
In particular, $f$ induces a map 
$f_m^\sharp: \mathrm{Gr}(m,\mathbb C^n)\cong G(m,n-m) \to G(N-n+m,n-m)\cong \mathrm{Gr}(N-n+m,\mathbb C^N)$
such that
\[ f((V+{\bf v})\cap \mathbb B^n) = \left(f_m^\sharp(V) +Jf({\bf 0})  {\bf v}\right) \cap f(\mathbb B^n) \subset \left(f_m^\sharp(V) +Jf({\bf 0}) {\bf v}\right)\cap \Omega \]
for any $V\in G(m,n-m)$ and ${\bf v}\in \mathbb B^n$, where $1\le m\le n-1$.
Note that the tangent map $df_{\bf 0}: \mathbb C^n\cong T_{\bf 0}(\mathbb B^n)\to T_{\bf 0}(\Omega) \cong \mathbb C^N$ of $f$ at ${\bf 0}$ is given by $df_{\bf 0}({\bf v}) = Jf({\bf 0}){\bf v}$.

\begin{Claim}
Let $V:=\ker {\bf A}$.
We have $f_m^\sharp(V):=\ker \big({\bf A} \overline{Jf({\bf 0})}^T\big)=df_{\bf 0}(V) \oplus \ker\overline{Jf({\bf 0})}^T$.
\end{Claim}

For any ${\bf v}\in V$, we have ${\bf A}\cdot {\bf v}={\bf 0}$ so that 
\[ {\bf A} \overline{Jf({\bf 0})}^T (df_{\bf 0}({\bf v}))
= {\bf A} \overline{Jf({\bf 0})}^TJf({\bf 0}) {\bf v}
= k {\bf A}  {\bf v} = {\bf 0}. \]
Thus, $df_{\bf 0}(V) \subset \ker \big({\bf A} \overline{Jf({\bf 0})}^T\big)$.
On the other hand, for any ${\bf z}\in \ker\overline{Jf({\bf 0})}^T$ we have $\overline{Jf({\bf 0})}^T {\bf z} = {\bf 0}$ so that ${\bf A} \overline{Jf({\bf 0})}^T{\bf z} = {\bf 0}$.
This yields $\ker\overline{Jf({\bf 0})}^T\subset \ker \big({\bf A} \overline{Jf({\bf 0})}^T\big)$.
For any ${\bf z} \in df_{\bf 0}(V) \cap \ker\overline{Jf({\bf 0})}^T$, we have ${\bf z} = Jf({\bf 0}) {\bf v}$ for some ${\bf v}\in V$ and $\overline{Jf({\bf 0})}^T {\bf z} = {\bf 0}$ so that 
\[ k {\bf v} = \overline{Jf({\bf 0})}^TJf({\bf 0}) {\bf v} =\overline{Jf({\bf 0})}^T {\bf z} = {\bf 0} \]
by $\overline{Jf({\bf 0})}^T Jf({\bf 0})=k {\bf I}_n$,
i.e., ${\bf v}= {\bf 0}$ and ${\bf z} = Jf({\bf 0}) {\bf v} = {\bf 0}$.
In particular, we have $df_{\bf 0}(V) \cap \ker\overline{Jf({\bf 0})}^T=\{{\bf 0}\}$.
Therefore, $df_{\bf 0}(V) \oplus \ker\overline{Jf({\bf 0})}^T\subseteq \ker \big({\bf A} \overline{Jf({\bf 0})}^T\big)$.
But then we have $\dim_{\mathbb C} \left(df_{\bf 0}(V) \oplus \ker\overline{Jf({\bf 0})}^T\right) = N-n+m = \dim_{\mathbb C} \ker \big({\bf A} \overline{Jf({\bf 0})}^T\big)$. Hence, $f_m^\sharp(V)=\ker \big({\bf A} \overline{Jf({\bf 0})}^T\big)=df_{\bf 0}(V) \oplus \ker\overline{Jf({\bf 0})}^T$.
The claim is proved.

It remains to prove that $f_m^\sharp$ is injective.
If $f_m^\sharp(V)=f_m^\sharp(V')$ for some $V,V'\in G(m,n-m)$, then by considering $\overline{Jf({\bf 0})}^T ( f_m^\sharp(V))=\overline{Jf({\bf 0})}^T ( f_m^\sharp(V') )$ it is easy to check that $V'\subseteq V$ and similarly $V\subseteq V'$, i.e., $V=V'$.
The proof is complete.
\end{proof}
\begin{remark}\text{}
\begin{enumerate}
\item From Dor {\rm\cite{Do91}} and Proposition {\rm\ref{pro:Structure_Slicing}}, if a holomorphic isometry $f$ from $(\mathbb B^n,kg_{\mathbb B^n})$ to $ (\Omega,g_\Omega)$ with $f({\bf 0})={\bf 0}$ has certain nondegeneracy property, then $f$ is rational {\rm(}cf.\,Dor {\rm\cite[Theorem 1]{Do91}}{\rm)}.
However, it is known from Dor {\rm\cite{Do91}} that the family of holomorphic maps from $\mathbb B^n$ to $\mathbb C^N$ that take hyperplanes to hyperplanes is not that restrictive.
This is because for any holomorphic map $F:\mathbb B^n\to \mathbb C^N$ for $2<n<N$ given by $F(z)=(z_1,\ldots,z_n,F_{n+1}(z),\ldots,F_N(z))$ for $z=(z_1,\ldots,z_n)\in \mathbb B^n$ and for some holomorphic functions $F_j$ on $\mathbb B^n$, $n+1\le j\le N$, $F$ takes hyperplanes to hyperplanes.
On the other hand, it is known that there are holomorphic isometries $f:(\mathbb B^n,kg_{\mathbb B^n})\to (\Omega,g_\Omega)$ which are not rational for some integer $n,k$ with $n\ge 2$ and for some irreducible bounded symmetric domain $\Omega$ {\rm(}cf.\,Chan-Mok {\rm\cite{CM17}}{\rm)}.
\item From Rudin {\rm\cite[Theorem 8.1.2]{Rudin80}}, we have the linear map $Jf({\bf 0}):\mathbb C^n\to \mathbb C^N$ maps $\mathbb B^n$ into $\Omega$ so that $Jf({\bf 0}){\bf x} \in \Omega$ for any ${\bf x}\in \mathbb B^n \subset \mathbb C^n \cong T_{\bf 0}(\mathbb B^n)$.
\item Let $n$ and $m$ be positive integers such that $1\le m < n$.
We may identify $V\in \mathrm{Gr}(m;\mathbb C^n)$ in homogeneous coordinates as $[Z]\in M(n,m;\mathbb C)$ so that the column vectors of $Z$ are $\mathbb C$-linearly independent and span the vector space $V$.
In particular, we identify $[ZA]=[Z]$ for any invertible matrix $A\in M(m,m;\mathbb C)$.

Identifying $\ker \overline{Jf({\bf 0})}^T \in \mathrm{Gr}(N-n;\mathbb C^N)$ as $[Z_0]\in M(N,N-n;\mathbb C)$, we have
\[ f_m^\sharp([Z]) = \left[ Jf({\bf 0}) Z, Z_0 \right] \]
in terms of the homogeneous coordinates.
This is well-defined as we have
\[ \begin{split}
f_m^\sharp([ZA]) 
=& \left[ Jf({\bf 0}) ZA, Z_0 \right]
= \left[ \begin{bmatrix} Jf({\bf 0}) Z, Z_0 \end{bmatrix} \begin{bmatrix}
A & 0\\
0 & {\bf I}_{N-n}
\end{bmatrix} \right] \\
=&\left[  Jf({\bf 0}) Z, Z_0\right].\end{split}\]
\end{enumerate}
\end{remark}

\section{Existence of holomorphic isometries whose images are the graphs of certain holomorphic maps}
Let $\Omega\Subset \mathbb C^N$ be an irreducible bounded symmetric domain of rank $\ge 2$ in its Harish-Chandra realization.
Recall that we have the compact dual Hermitian symmetric space $X_c$ of $\Omega$ and the first canonical embedding (also called the \emph{minimal embedding}) $X_c\hookrightarrow \mathbb P\big(\Gamma(X_c,\mathcal O(1))^*\big)\cong \mathbb P^{N'}$.
It is natural to ask if there exists a holomorphic isometry $f:(\mathbb B^n,g_{\mathbb B^n})\to (\Omega,g_\Omega)$ such that $f(\mathbb B^n)=\mathrm{Graph}(\widetilde f)$ for some holomorphic map $\widetilde f:\mathbb B^n \to \mathbb C^{N-n}$.
In particular, we study the existence of a holomorphic isometry $f:(\mathbb B^n,g_{\mathbb B^n})\to (\Omega,g_\Omega)$ such that $f({\bf 0})={\bf 0}$ and of the form
\[ f(z)=(z_1,\ldots,z_n,f_{n+1}(z),\ldots,f_N(z)) \]
for $z=(z_1,\ldots,z_n)\in \mathbb B^n$ and some integer $n\ge 2$.
The existence of such an isometry $f$ would imply that
\[ 1-\sum_{j=1}^n |z_j|^2 - \sum_{j=n+1}^N |f_j(z)|^2 + \sum_{l=1}^{N'-N} (-1)^{\deg G_l} |G_l(f(z))|^2 = 1-\sum_{j=1}^n |z_j|^2 \]
so that
\[  \sum_{l=1}^{N'-N} (-1)^{\deg G_l} |G_l(f(z))|^2 = \sum_{j=n+1}^N |f_j(z)|^2, \]
where $G_l$'s are homogeneous polynomials (depending on $\Omega$) as in Equation (\ref{Eq:sp_form_poly_h}).
This will be useful for constructing new holomorphic isometries from $(\mathbb B^n,g_{\mathbb B^n})$ to $(\Omega,g_\Omega)$.

Motivated by the study in Chan \cite{Ch18}, we consider the simplest nontrivial case where $\Omega$ is of rank $2$ and satisfies $2N-N'>1$, where $N$ and $N'$ are integers defined above.
Define
\[ W_{{\bf U'}}:=
\left\{ w=(w_1,\ldots,w_N)\in \Omega: {\bf U}' \begin{pmatrix}
w_1\\\vdots \\ w_N
\end{pmatrix}= \begin{pmatrix}
G_{1}(w)\\\vdots \\ G_{N'-N}(w)
\end{pmatrix} \right\} \]
for any ${\bf U'}\in M(N'-N,N;\mathbb C)$ of full rank $N'-N$ (cf.\,Chan \cite{Ch18}).
We have the following existence theorem for holomorphic isometries $f:(\mathbb B^n,g_{\mathbb B^n})\to (\Omega,g_\Omega)$, $n\le 2N-N'$, such that $f(\mathbb B^n)$ are the graphs of certain holomorphic maps.

\begin{theorem}\label{Thm:Existence_SForm}
Let $\Omega\Subset \mathbb C^N $ be an irreducible bounded symmetric domain of rank $2$ in its Harish-Chandra realization.
Let $X_c$ be the compact dual of $\Omega$ and $X_c\hookrightarrow \mathbb P\big(\Gamma(X_c,\mathcal O(1))^*\big)\cong \mathbb P^{N'}$ be the first canonical embedding {\rm(}also called the minimal embedding{\rm)}.
Suppose the rank-$2$ irreducible bounded symmetric domain $\Omega$ satisfies $2N-N'>1$, equivalently $\Omega\not \cong D^{\mathrm{I}}_{2,q}$ for any $q\ge 5$ \rm{(cf.\,Chan \cite[Lemma 4.1]{Ch18})}.
Given any $\widetilde{\bf U}\in U(N'-N)$, let ${\bf U'}\in M(N'-N,N;\mathbb C)$ be such that ${\bf U'}=\begin{bmatrix}
{\bf 0} & \widetilde{\bf U}
\end{bmatrix}$.
Then, the irreducible component of $W_{\bf U'}$ containing ${\bf 0}$ is the image of a holomorphic isometry $f:(\mathbb B^{2N-N'},g_{\mathbb B^{2N-N'}})\to (\Omega,g_\Omega)$ given by
\[ f(z)=(z_1,\ldots,z_{2N-N'},f_{2N-N'+1}(z),\ldots,f_N(z)) \]
for some holomorphic functions $f_j:\mathbb B^{2N-N'}\to \mathbb C$, $2N-N'+1 \le j\le N$, satisfying
\[ \widetilde{\bf U}\begin{pmatrix}
f_{2N-N'+1}(z)\\
\vdots\\
f_N(z)
\end{pmatrix}=\begin{pmatrix}
G_{1}(z,f_{2N-N'+1}(z),\ldots,f_N(z))\\
\vdots\\
G_{N'-N}(z,f_{2N-N'+1}(z),\ldots,f_N(z))
\end{pmatrix}, \]
where $z=(z_1,\ldots,z_{2N-N'})\in \mathbb B^{2N-N'}$.
In particular, for any $1\le n \le 2N-N'$ we obtain a holomorphic isometry $\Psi:(\mathbb B^n,g_{\mathbb B^n})\to (\Omega,g_\Omega)$ given by
$\Psi(w)=(w,\psi(w))$
for $w\in \mathbb B^n$, where $\psi:\mathbb B^n\to \mathbb C^{N-n}$ is some holomorphic map.
\end{theorem}
\begin{proof}
Note that the K\"ahler form $\omega_{g_\Omega}$ corresponding to $g_\Omega$ may be written as
\[ \omega_{g_\Omega}=-\sqrt{-1}\partial\overline\partial \log h_\Omega(w,w), \]
where $h_\Omega(w,w)=1-\sum_{j=1}^N|w_j|^2+\sum_{l=1}^{N'-N} |G_l(w)|^2$ for $w\in \mathbb C^N$ and $G_l(w)$, $1\le l\le N'-N$, are homogeneous polynomials of degree $2$ in $w\in \mathbb C^N$ (cf.\,Chan \cite{Ch18}).
Since ${\bf U'}=\begin{bmatrix}
{\bf 0} & \widetilde {\bf U}
\end{bmatrix}$, the system of equations
\[ {\bf U}' \begin{pmatrix}
w_1,\ldots, w_N
\end{pmatrix}^T= \begin{pmatrix}
G_{1}(w),\ldots,G_{N'-N}(w)
\end{pmatrix}^T \]
is the same as
\[ \begin{pmatrix}
w_{2N-N'+1},\ldots,
w_{N}
\end{pmatrix}^T={\bf U} \begin{pmatrix}
G_{1}(w),\ldots,G_{N'-N}(w)
\end{pmatrix}^T, \]
where ${\bf U}:=\overline{\widetilde{\bf U}}^T\in U(N'-N)$.
Let $F_1,\dots,F_{N'-N}$ be holomorphic functions on $\mathbb C^N$ defined by
\[ \begin{pmatrix}
F_1(w)\\\vdots \\ F_{N'-N}(w)
\end{pmatrix}
:=\begin{pmatrix}
w_{2N-N'+1}\\
\vdots\\
w_{N}
\end{pmatrix}-{\bf U} \begin{pmatrix}
G_{1}(w)\\
\vdots\\
G_{N'-N}(w)
\end{pmatrix}. \]
Since $\det\begin{pmatrix}
{\partial F_i\over \partial w_{2N-N'+j}}({\bf 0})
\end{pmatrix}_{1\le i,j\le N'-N}=\det {\bf I}_{N'-N}\neq 0$, by the Implicit Function Theorem there exists a germ of holomorphic map $\phi:=(\phi_1,\ldots,\phi_{N'-N})$ $:$ $(\mathbb C^{2N-N'};{\bf 0})$ $\to$ $(\mathbb C^{N'-N};{\bf 0})$ such that
$F_j(z,\phi(z)) = 0$ for $1\le j\le N'-N$, $z=(z_1,\ldots,z_{2N-N'})\in \mathbb C^{2N-N'}$,
i.e., 
\[ \begin{pmatrix}
\phi_1(z)\\
\vdots\\
\phi_{N'-N}(z)
\end{pmatrix}={\bf U} \begin{pmatrix}
G_{1}(z,\phi(z))\\
\vdots\\
G_{N'-N}(z,\phi(z))
\end{pmatrix}. \]
This implies that
$\displaystyle\sum_{j=1}^{N'-N} |\phi_j(z)|^2
= \sum_{l=1}^{N'-N} |G_{l}(z,\phi(z))|^2$
so that
\[ 1-\sum_{j=1}^{2N-N'}|z_j|^2 - \sum_{j=1}^{N'-N} |\phi_j(z)|^2
+ \sum_{l=1}^{N'-N} |G_{l}(z,\phi(z))|^2
=1-\sum_{j=1}^{2N-N'}|z_j|^2, \]
i.e., $h_\Omega(f(z),f(z))=1-\sum_{j=1}^{2N-N'}|z_j|^2$, where $f:(\mathbb C^{2N-N'};{\bf 0})\to (\mathbb C^N;{\bf 0})$ is the germ of holomorphic map defined by $f(z)=(z,\phi(z))$.
In particular, $f:(\mathbb B^{2N-N'},g_{\mathbb B^{2N-N'}};{\bf 0})\to (\Omega,g_\Omega;{\bf 0})$ is a germ of holomorphic isometry, which extends to a holomorphic isometry from $(\mathbb B^{2N-N'},g_{\mathbb B^{2N-N'}})$ to $(\Omega,g_\Omega)$ by the extension theorem of Mok \cite{Mok12}.
The extension is also denoted by $f$.
It is clear that $f(\mathbb B^{2N-N'})\subset W_{\bf U'}$ and ${\bf 0}\in f(\mathbb B^{2N-N'})$.
Since ${\bf U'}\overline{{\bf U'}}^T= {\bf I}_{N'-N}$, it follows from Chan \cite[Proof of Proposition 4.3, p.\,311]{Ch18} that the irreducible component of $W_{\bf U'}$ containing ${\bf 0}$ is $f(\mathbb B^{2N-N'})$ and $f$ has the desired form by the above constructions. (Noting that ${\bf 0}$ is a regular point of $W_{\bf U'}$ so that there exists a unique irreducible component of $W_{\bf U'}$ containing ${\bf 0}$.)

Now, let $n$ be an integer such that $1\le n \le 2N-N'$.
If $n = 2N-N'$, then we may put $\Psi:=f$ and we are done. 
Suppose $n< 2N-N'$.
By putting $z_{n+1}=\cdots = z_{2N-N'}=0$, we have a holomorphic isometry $\Psi:(\mathbb B^n,g_{\mathbb B^n})\to (\Omega,g_\Omega)$ given by
\[ \Psi(w_1,\ldots,w_n):=f(w_1,\ldots,w_n,{\bf 0})
= (w_1,\ldots,w_n,{\bf 0}, \phi(w_1,\ldots,w_n,{\bf 0})) \]
for $(w_1,\ldots,w_n)\in \mathbb B^n$, as desired.
\end{proof}
\begin{remark}
In general, given any $\widetilde{\bf U}=\begin{bmatrix}
{\bf u}_1 & \cdots & {\bf u}_{N'-N}
\end{bmatrix}\in U(N'-N)$, let ${\bf U'}=\begin{bmatrix}
{\bf u}'_1 &\cdots & {\bf u}'_N
\end{bmatrix} \in M(N'-N,N;\mathbb C)$ be such that
${\bf u}'_{\sigma(j)}={\bf u}_j$, $j=1,\ldots,N'-N$,
and
${\bf u}'_{\sigma(N'-N+l)} = {\bf 0}$, $l=1,\ldots,2N-N'$, for some permutation $\sigma\in \Sigma_N$.
Moreover, let $\Omega$ be as in the statement of Theorem {\rm\ref{Thm:Existence_SForm}}.
Then, from the proof of Theorem {\rm\ref{Thm:Existence_SForm}}, the irreducible component of $W_{\bf U'}$ containing ${\bf 0}$ is the image of some holomorphic isometry $f:(\mathbb B^{2N-N'},g_{\mathbb B^{2N-N'}})\to (\Omega,g_\Omega)$ given by
$f(z)=(f_1(z),\ldots,f_N(z))$ with
$f_{\sigma(N'-N+l)}(z)\equiv z_l$ for $l=1,\ldots,2N-N'$.
\end{remark}

\section{Induced surjective holomorphic submersion}
We will show that given any holomorphic isometric embedding $f$ $:$ $(\mathbb B^n,kg_{\mathbb B^n})$ $\to$ $(\Omega,g_\Omega)$ such that $f({\bf 0})={\bf 0}$, $f$ is a global holomorphic section of a certain surjective holomorphic submersion from $D_f$ to $\mathbb B^n$, where $D_f\Subset \mathbb C^N$ is a bounded balanced convex domain such that $D_f\subseteq \Omega$ and the image of the holomorphic isometric embedding $f$ lies inside $D_f$. More precisely, we have

\begin{theorem}\label{thm:HoloSubmersion1}
Let $f:(\mathbb B^n,kg_{\mathbb B^n})\to (\Omega,g_\Omega)$ be a holomorphic isometric embedding such that $f({\bf 0})={\bf 0}$, where $\Omega\Subset \mathbb C^N$ is an irreducible bounded symmetric domain in its Harish-Chandra realization and $k>0$ is a positive real constant. Define 
\[ D_f:=\{z\in \Omega\mid \lVert z \rVert_f < k \}, \]
where $\langle z,\overline\xi \rangle_f:= \overline{\xi}^T Jf({\bf 0}) \overline{Jf({\bf 0})}^T z$ is a complex Hermitian form defined on $\mathbb C^N$ for $z,\xi\in \mathbb C^N$ being identified as column vectors, and $\lVert z\rVert_f^2:= \langle z,\overline{z} \rangle_f = \lVert \overline{Jf({\bf 0})}^T z\rVert_{\mathbb C^n}^2$.
Then, $D_f\Subset \mathbb C^N$ is a bounded balanced convex domain and
$f(\mathbb B^n)\subset D_f$.
In particular, there is a holomorphic map $\hat f:\mathbb B^n \to D_f$ such that $f=\iota\circ \hat f$, where $\iota:D_f\hookrightarrow \Omega$ is the inclusion map.
Moreover, there is a surjective holomorphic submersion $\widetilde\pi_f:D_f \to \mathbb B^n$ such that 
$\widetilde\pi_f\circ \hat f = \id_{\mathbb B^n}$,
i.e., $\hat f:\mathbb B^n \to D_f$ is a global holomorphic section of the surjective holomorphic submersion $\widetilde\pi_f:D_f \to \mathbb B^n$.
\end{theorem}
\begin{proof}
Write $f=(f_1,\ldots,f_N)$.
From the assumption, we have
\[ \overline{Jf({\bf 0})}^T \begin{pmatrix}
f_1(w)\\\vdots \\ f_N(w)
\end{pmatrix}
= k \begin{pmatrix} w_1 \\ \vdots \\ w_n \end{pmatrix} \]
for any $(w_1,\ldots,w_n)\in \mathbb B^n$.
Let $\mathcal L_f:\mathbb C^N\to \mathbb C^n$ be the linear map
\[ \mathcal L_f(z):= {1\over k}\overline{Jf({\bf 0})}^T \begin{pmatrix}
z_1,\ldots,z_N
\end{pmatrix}^T. \]
Let $\pi_f:=\mathcal L_f|_\Omega:\Omega\to \mathbb C^n$ be the restriction of $\mathcal L_f$ to $\Omega\Subset \mathbb C^N$.
Then, $\pi_f(f(\mathbb B^n))=\mathbb B^n$.
Consider the preimage $\pi_f^{-1}(\mathbb B^n) =\left\{z\in \Omega \mid \lVert \pi_f(z) \rVert^2_{\mathbb C^n} <1 \right\}$.
Then, we have
\[ \pi_f^{-1}(\mathbb B^n) =\{z\in \Omega\mid \lVert z \rVert_f < k \}=:D_f. \]
Then, we have a complex manifold $D_f\subseteq \Omega$.
Actually, $D_f = \mathcal L_f^{-1}(\mathbb B^n)\cap \Omega$ is a bounded open subset of $\mathbb C^N$.

\begin{Claim}
The subset $D_f\subset \Omega\Subset \mathbb C^N$ is connected and convex. In particular, $D_f\subset \mathbb C^N$ is a bounded balanced convex domain.
\end{Claim}

Let $z,\xi\in D_f$ be any two distinct points.
Consider the line segment joining $z$ and $\xi$ given by $\gamma(t)=(1-t) z + t \xi$, $0\le t\le 1$.
Since $\Omega\Subset \mathbb C^N$ is a bounded convex domain and $z,\xi\in \Omega$, $\gamma(t)\in \Omega$ for all $t\in [0,1]$.
Moreover, regarding $z,\xi\in \mathbb C^N\cong M(N,1;\mathbb C)$ as column vectors, we have
$\lVert (1-t) z + t \xi \rVert_f^2
= ((1-t) \overline{z}^T + t \overline{\xi}^T) Jf({\bf 0})\overline{Jf({\bf 0})}^T ((1-t) z + t \xi)
= (1-t)^2 \lVert z\rVert_f^2
+ 2 t (1-t) \mathrm{Re} \langle z,\overline\xi\rangle_f
+ t^2 \lVert \xi\rVert_f^2$.
Note that
\[\begin{split}
 |\langle z,\overline\xi\rangle_f|
=& \left|\left\langle \overline{Jf({\bf 0})}^T z, \overline{\overline{Jf({\bf 0})}^T \xi}\right\rangle_{\mathbb C^n}\right|\\
\le& \lVert  \overline{Jf({\bf 0})}^T z\rVert_{\mathbb C^n}
\lVert  \overline{Jf({\bf 0})}^T \xi \rVert_{\mathbb C^n}\;\;(\text{By the Cauchy-Schwarz inequality})\\
=& \lVert z \rVert_f \lVert \xi \rVert_f. 
\end{split}\]
Therefore, for $0\le t\le 1$ we have
$\lVert (1-t) z + t \xi \rVert_f^2 \le (1-t)^2 \lVert z\rVert_f^2
+ 2 t (1-t) \lVert z \rVert_f \lVert \xi \rVert_f
+ t^2 \lVert \xi\rVert_f^2
= \left((1-t) \lVert z\rVert_f+ t \lVert \xi\rVert_f\right)^2$
so that 
\[ \lVert (1-t) z + t \xi \rVert_f  \le (1-t) \lVert z\rVert_f+ t \lVert \xi\rVert_f
< (1-t)k+ tk = k. \]
This shows that $(1-t) z + t \xi \in D_f$ for $0\le t\le 1$.
In other words, $D_f$ is convex.
Since $\mathbb B^n$ and $\Omega$ are balanced domains, so is $D_f$ from the construction.
Hence, $D_f$ is a bounded balanced convex domain in $\mathbb C^N$. The claim is proved.

Since $\pi_f(D_f)\subseteq \mathbb B^n$, we let $\widetilde\pi_f:D_f\to \mathbb B^n$ be the restriction of $\pi_f$ to $D_f$.
Note that $f(\mathbb B^n)\subset D_f$ and $\widetilde\pi_f(f(\mathbb B^n))=\mathbb B^n$ so that $\widetilde\pi_f$ is surjective.
Moreover, the rank of $d\widetilde\pi_f$ is equal $\mathrm{rank}(\overline{Jf({\bf 0})}^T)=n$ so that $\widetilde\pi_f:D_f\to \mathbb B^n$ is a surjective holomorphic submersion.
Now, it is clear that there is such a holomorphic map $\hat f:\mathbb B^n\to D_f$ as in the statement of the theorem. Then, $\hat f$ is a global holomorphic section of $\widetilde\pi_f$ in the sense that $\widetilde\pi_f\circ \hat f = \mathrm{Id}_{\mathbb B^n}$.
\end{proof}
\begin{remark}\text{}
\begin{enumerate}
\item A subset $U\subset \mathbb C^N$ is said to be balanced if and only if $\lambda z\in U$ whenever $z\in U$ and $\lambda\in \mathbb C$ such that $|\lambda|\le 1$ {\rm(}cf.\,Rudin {\rm\cite[p.\,161]{Rudin80}}{\rm)}.
\item $D_f\Subset \mathbb C^N$ is a domain of holomorphy and thus a Stein manifold.
\item
For $n\ge 3$, let $f:(\mathbb B^{n-1},g_{\mathbb B^{n-1}})\to (D^{\mathrm{IV}}_n,g_{D^{\mathrm{IV}}_n})$ be the holomorphic isometry given by 
\[ f(w_1,\ldots,w_{n-1})=\left(w_1,\ldots,w_{n-1},1-\sqrt{1-\sum_{j=1}^{n-1} w_j^2}\right),\]
we can check that $D_f=\left\{(z_1,\ldots,z_n)\in D^{\mathrm{IV}}_n: \sum_{j=1}^{n-1}|z_j|^2<1\right\} \subsetneq D^{\mathrm{IV}}_n$.
\item Let $f:(\mathbb B^{2N-N'},g_{\mathbb B^{2N-N'}})\to (\Omega,g_\Omega)$ be the holomorphic isometry defined in Theorem {\rm\ref{Thm:Existence_SForm}}.
Then, we have
\[ \begin{split}
D_f &=\left\{(z_1,\ldots,z_N)\in \Omega: \sum_{j=1}^{2N-N'} |z_j|^2<1\right\}\\
&=\Omega\cap (\mathbb B^{2N-N'}\times \mathbb C^{N'-N}). 
\end{split}\]
\item Defining
$D_{f,r}:=\{z\in \Omega\mid \lVert z \rVert_f < kr \}$
for $r\in (0,1)$, from the proof of Theorem {\rm\ref{thm:HoloSubmersion1}} we know that $D_{f,r}\subset \mathbb C^N$ is a bounded convex domain and $D_{f,r}\subset D_f \subset \Omega$.
In addition, for $r\in (0,1)$ we have $f(\mathbb B^n({\bf 0},r))= D_{f,r}\cap f(\mathbb B^n)$.
In general, for any $z_0\in \Omega$ and $r>0$, we write 
\[ D_f(z_0,r):= \left\{z\in \Omega: \lVert z-z_0 \rVert_f < kr\right\}. \]
Then, for any $r>0$ and $w_0\in \mathbb B^n$ such that $\mathbb B^n(w_0,r)\Subset \mathbb B^n$, we have $f(\mathbb B^n(w_0,r))=D_f(f(w_0),r) \cap f(\mathbb B^n)$.
\end{enumerate}
\end{remark}

\section{Geometric properties}
In this section, by combining our previous observations and some classical results, we have found more geometric properties of holomorphic isometries between bounded symmetric domains.

\subsection{Holomorphic retractions}
\begin{proposition}
Let $f:(\mathbb B^n,kg_{\mathbb B^n})\to (\Omega,g_\Omega)$ be a holomorphic isometric embedding such that $f({\bf 0})={\bf 0}$, where $\Omega\Subset \mathbb C^N$ is an irreducible bounded symmetric domain of rank $\ge 2$ in its Harish-Chandra realization and $k>0$ is a positive real constant.
Let $D_f$ and $\widetilde\pi_f:D_f\to \mathbb B^n$ be as defined in Theorem {\rm\ref{thm:HoloSubmersion1}}.
Define $r_f:=f\circ \widetilde\pi_f: D_f \to f(\mathbb B^n)=:S \subset D_f$.
Then, we have the following.
\begin{enumerate}
\item $r_f$ is a holomorphic retraction,
\item $r_f^{-1}(x)$ is an affine linear section of the bounded convex domain $D_f$ for all $x\in S$,
\item $z-r_f(z)\in r_f^{-1}(0)$ for all $z\in D_f$.
\end{enumerate}
\end{proposition}
\begin{proof}
Write $\iota:S\hookrightarrow D_f$ for the inclusion map.
For any $x=f(w)\in S=f(\mathbb B^n)$, $w\in \mathbb B^n$, we have
\[ r_f \circ \iota(x) = f(\widetilde\pi_f(f(w))) = f(w)=x, \]
i.e., $r_f\circ \iota = {\rm Id}_S$ and thus $r_f:D_f\to S$ is a holomorphic retraction.
This proves (1).
We choose a point $x=f(w)\in S$, $w=(w_1,\ldots,w_n)\in \mathbb B^n$.
Since $f:\mathbb B^n \to \Omega$ is injective, we have
\[ \begin{split}
r_f^{-1}(x) &= \{ z\in D_f: f(\widetilde\pi_f(z))=x = f(w) \} = \{z\in D_f: \widetilde\pi_f(z)=w\}\\
&= \left\{z\in D_f: {1\over k} \cdot (z_1,\ldots,z_N) \overline{Jf({\bf 0})} = (w_1,\ldots,w_n)\right\} 
\end{split},\]
which is an affine linear section of the bounded convex domain $D_f$. This proves (2).

Viewing $r_f$ as a holomorphic self-map of $D_f$, $S=\{z\in D_f:r_f(z)=z\}$ is the fixed point set of $r_f:D_f\to D_f$.
We also have
$\widetilde \pi_f(z-r_f(z)) = \widetilde \pi_f(z - f(\widetilde \pi_f(z)))
=\widetilde \pi_f(z)-(\widetilde \pi_f\circ f)(\widetilde \pi_f(z)))
=\widetilde \pi_f(z)-\widetilde \pi_f(z)
=0$ so that
\[ r_f(z-r_f(z)) = 0 \quad \forall\; z\in D_f \]
and thus $z-r_f(z)\in r_f^{-1}(0)$ for all $z\in D_f$, and (3) follows.
\end{proof}
\begin{remark}
In the particular case where $f$ is totally geodesic, it follows from Mok {\rm\cite[Theorem 1.1]{Mok22}} that there is a natural holomorphic retraction of $\Omega$ onto $S=f(\mathbb B^n)$.
On the other hand, we have seen in Remark 3.3 that $D_f \subsetneq \Omega$ in general.
\end{remark}

\subsection{On complete intersections}
In this section, we follow B\v{a}nic\v{a}-Forster \cite{BF82} for the notions of local complete intersections, ideal-theoretic complete intersections (also called complete intersections) and set-theoretic complete intersections.
Combining with some classical results for (closed) complex-analytic subvarieties of Stein manifolds, we have the following result.

\begin{theorem}\label{thm_CI1}
Let $f:(\mathbb B^n,kg_{\mathbb B^n})\to (\Omega,g_\Omega)$ be a holomorphic isometric embedding, where $k>0$ is a real constant and $\Omega\Subset \mathbb C^N$ is an irreducible bounded symmetric domain.
Write $S:=f(\mathbb B^n)$.
Then, the normal bundle $N_{S|\Omega}=T_\Omega|_S/T_S$ to $S$ in $\Omega$ is trivial, $S$ is a leaf in a nonsingular holomorphic foliation of $\Omega$, and $S$ is a set-theoretic complete intersection, i.e., there exist $N-n$ holomorphic functions $g_1,\ldots,g_{N-n}$ on $\Omega$ such that
$S=\{z\in \Omega: g_j(z)=0\;{\rm for}\;j=1,\ldots,N-n\}$.
\end{theorem}
\begin{proof}
Since $\Omega$ is homogeneous, we may suppose without loss of generality that $f({\bf 0})={\bf 0}$.
By Theorem \ref{thm:HoloSubmersion1}, we have a holomorphic submersion $\widetilde\pi_f:D_f\to \mathbb B^n$ from a bounded convex domain $D_f\subset \Omega$ onto $\mathbb B^n$ such that $\widetilde\pi_f\circ f = \id_{\mathbb B^n}$, and we have
\[ T_{f(w)}(\Omega) =T_{f(w)}(S)\oplus \ker d\widetilde \pi_f(f(w))\quad\forall\;w\in \mathbb B^n. \]
This yields a holomorphic vector subbundle 
\[ \mathcal N_S:=\bigcup_{x\in S} \ker d\widetilde \pi_f(x)=(\ker d\widetilde\pi_f)|_S\]
of $T_\Omega|_S=T_{D_f}|_S$.
(Noting that $\ker d\widetilde \pi_f(x) = T_x (\widetilde \pi_f^{-1}(w))$ for all $x\in \widetilde \pi_f^{-1}(w)$, so that $\ker d\widetilde \pi_f(f(w)) = T_{f(w)} (\widetilde \pi_f^{-1}(w))$.)
Let $\pi_f:\Omega \to \mathbb C^n$ be as in the proof of Theorem \ref{thm:HoloSubmersion1} so that $\widetilde \pi_f=\pi_f|_{D_f}$.
Then, we have a (trivial) holomorphic vector subbundle $N:=\bigcup_{z\in \Omega} \ker d\pi_f(z)$ of $T_\Omega$ such that the total space of $N$ can be identified with $\Omega \times \ker \overline{Jf({\bf 0})}^T$ via the trivialization $T_\Omega\cong \Omega \times \mathbb C^N$, and we have $\mathcal N_S=N|_S$ so that $T_\Omega|_S = T_S \oplus N|_S$. (Noting that this direct sum decomposition is actually a holomorphic splitting by Theorem \ref{thm_CI_general}.)

Since $\Omega$ is Stein, by Forstneri\v{c} \cite[Corollary 7.2 and the remark after its proof, pp.\,185--186]{For03} $S=f(\mathbb B^n)$ is a leaf in a nonsingular holomorphic foliation of $\Omega$, and the normal bundle $N_{S|\Omega}$ to $S$ in $\Omega$ admits locally constant transition functions, i.e., $N_{S|\Omega}$ is a flat bundle, and thus $N_{S|\Omega}$ is a trivial bundle since $S\cong \mathbb B^n$ is simply connected (cf.\,Forstneri\v{c} \cite[p.\,182]{For03}).
The flatness of $N_{S|\Omega}$ implies that the conormal bundle $N_{S|\Omega}^*$ of $S$ is also a flat bundle, and hence $N_{S|\Omega}^*$ is a trivial bundle by the simply connectedness of $S$.
Moreover, the (closed) complex submanifold $S\subset \Omega$ is clearly a local complete intersection in the contractible Stein manifold $\Omega$.
By a classical result of Boraty\'nski \cite{Bo78} and B\v{a}nic\v{a}-Forster \cite[1.6.\,Theorem]{BF82}, $S$ is a set theoretic complete intersection, i.e., there exist $N-n$ holomorphic functions $g_1,\ldots,g_{N-n}$ on $\Omega$ such that 
$S=\{z\in \Omega: g_j(z)=0\,\text{for}\,j=1,\ldots,N-n\}$.
\end{proof}

From the classical Oka-Grauert theory, every holomorphic vector bundle on a contractible Stein space is holomorphically trivial (cf.\,Forstneri\v{c} \cite[p.\,244]{For17}). As a consequence, we have the following result of B\v{a}nic\v{a}-Forster \cite{BF82}.
\begin{proposition}[B\v{a}nic\v{a}-Forster \cite{BF82}]\label{pro:CI}
Let $f:D\to \Omega$ be a proper holomorphic embedding from a contractible complex manifold $D$ into a contractible Stein space $\Omega$.
Define $S:=f(D)\subset \Omega$.
Then, $S$ is a set-theoretic complete intersection.
If in addition that $\dim_{\mathbb C}(S)\le {2(\dim_{\mathbb C}(\Omega)-1)\over 3}$, then $S$ is an ideal-theoretic complete intersection.
\end{proposition}
\begin{proof}
By the Steinness of $\Omega$, $S=f(D)\subset \Omega$ is a (closed) contractible Stein submanifold.
In particular, the (holomorphic) normal bundle $N_{S|\Omega}=T_\Omega|_S/T_S$ is trivial from the classical Oka-Grauert theory.
By the smoothness of $S$, $S \subset \Omega$ is a local complete intersection.
Hence, $S$ is a set-theoretic complete intersection by a classical result of Boraty\'nski \cite{Bo78} and B\v{a}nic\v{a}-Forster \cite[1.6.\,Theorem]{BF82}.
When $\dim_{\mathbb C}(S)\le {2(\dim_{\mathbb C}(\Omega)-1)\over 3}$, $S$ is an ideal-theoretic complete intersection by B\v{a}nic\v{a}-Forster \cite[1.5.\,Corollary]{BF82}.
\end{proof}

On the other hand, Theorem \ref{thm:HoloSubmersion1} can be easily generalized to the case of holomorphic isometries between bounded symmetric domains with respect to the canonical K\"ahler metrics, as follows.

\begin{theorem}\label{thm_CI_general}
Let $f$ $:$ $(D_1,\lambda_1 g_{D_1})$ $\times$ $\cdots$ $\times$ $(D_k,\lambda_k g_{D_k})$ $\to$ $(\Omega_1,\mu_1 g_{\Omega_1})$ $\times$ $\cdots$ $\times$ $(\Omega_m,\mu_m g_{\Omega_m})$ be a holomorphic isometric embedding, where $k,m$ are positive integers, $D_1,\ldots,D_k,\Omega_1,\ldots,\Omega_m$ are irreducible bounded symmetric domains in complex Euclidean spaces,
and $\lambda_i>0$, $1\le i\le k$, $\mu_j>0$, $1\le j\le m$, are real constants.
Write $D:=D_1\times\cdots \times D_k$, $\Omega:=\Omega_1\times\cdots \times \Omega_m$, $S:=f(D)$, $n:=\dim_{\mathbb C}(D)$ and $N:=\dim_{\mathbb C}(\Omega)$.

Suppose $f({\bf 0})={\bf 0}$.
Then, there exist an $n$-by-$N$ matrix $M(f)$ of rank $n$, a bounded balanced convex domain $D_f \subset \Omega \subset \mathbb C^N$, and a holomorphic submersion $\widetilde \pi_f:D_f\to D$ defined by $\widetilde\pi_f(z_1,\ldots,z_N):=(z_1,\ldots,z_N) M(f)^T$ for $(z_1,\ldots,z_N)\in D_f$ such that $S=f(D)\subset D_f$, and $\widetilde \pi_f(f(w))=w$ for all $w\in D$.
In addition, we have
\[ T_{f(w)}(\Omega) = T_{f(w)}(S)\oplus \ker d\widetilde \pi_f(f(w)) \quad\forall\;w\in D, \]
and thus $T_\Omega|_S=T_S\oplus (\ker d\widetilde \pi_f)|_S$ is a holomorphic splitting.
In general, without assuming $f({\bf 0})={\bf 0}$, $f:D\to \Omega$ is a proper holomorphic isometric embedding so that $f(D)\subset \Omega$ is a closed embedded complex submanifold, and there still exist a holomorphic submersion $\widetilde \pi_f: D_f\to D$ from a bounded domain of holomorphy $D_f$ in $\mathbb C^N$ such that $f(D) \subset D_f \subset \Omega$, $\widetilde \pi_f\circ f={\rm Id}_D$, and $T_\Omega|_{S} = (\ker d\widetilde \pi_f)|_{S}\oplus T_{S}$ is a holomorphic splitting. As a consequence, the tangent sequence
\[ 0 \to T_S \xrightarrow{\iota_*} T_\Omega|_S \to N_{S|\Omega} \to 0 \]
also splits holomorphically over $S$.
\end{theorem}
\begin{proof}
Suppose $f({\bf 0})={\bf 0}$.
Write $n_i:=\dim_{\mathbb C}(D_i)$ for $1\le i\le k$, and $N_j:=\dim_{\mathbb C}(\Omega_j)$ for $1\le j\le m$.
Write $f=(f_1,\ldots,f_m)$ so that $f_j=(f_j^1,\ldots,f_j^{N_j}):D\to \Omega_j$ is a holomorphic map for $1\le j\le m$.
The isometric condition of $f:D\to \Omega$ means that
\[ \sum_{j=1}^m\mu_jf_j^*\omega_{g_{\Omega_j}}
=\sum_{i=1}^k \lambda_i {\rm Pr}_i^*\omega_{g_{D_i}}, \]
where ${\rm Pr}_i:D\to D_i$ is the canonical projection onto the $i$-th irreducible factor $D_i$ of $D$.
Since $f({\bf 0})={\bf 0}$, we have the polarized functional equation
\[ \prod_{j=1}^m h_{\Omega_j}(f_j(w),f_j(\zeta))^{\mu_j}
=\prod_{i=1}^k h_{D_i}(w^i,\zeta^i)^{\lambda_i} \]
for $w=(w^1,\ldots,w^k),\zeta=(\zeta^1,\ldots,\zeta^k)\in D_1\times\cdots \times D_k$.
Write $w^i=(w^i_1,\ldots,w^i_{n_i})$ and $\zeta^i=(\zeta^i_1,\ldots,\zeta^i_{n_i})$ for $1\le i\le k$.
Differentiating the polarized functional equation with respect to $\overline{\zeta^i_\nu}$ for $1\le \nu\le n_i$ and $1\le i\le k$, and evaluating at $\zeta={\bf 0}$, we have
\begin{equation} \label{Eq:Diff_PFE_Gen}
\sum_{j=1}^m \sum_{\alpha=1}^{N_j} \mu_j f_j^\alpha(w) \overline{{\partial f_j^\alpha\over \partial\zeta_\nu^i}({\bf 0})}
=\lambda_i w_\nu^i
\end{equation}
for $1\le \nu\le n_i$ and $1\le i\le k$ by using the fact that each $h_{D_i}$ (resp.\,$h_{\Omega_j}$) is of the form in (\ref{Eq:sp_form_poly_h}).
By direct computations, from Equation (\ref{Eq:Diff_PFE_Gen}) we have
\[ \begin{bmatrix}
{1\over \lambda_1}{\bf I}_{n_1} & & \\
& \ddots & \\
 & & {1\over \lambda_k} {\bf I}_{n_k}
 \end{bmatrix}
\overline{Jf({\bf 0})}^T
 \begin{bmatrix}
\mu_1{\bf I}_{N_1} & & \\
& \ddots & \\
 & & {\mu_m} {\bf I}_{N_m}
 \end{bmatrix}
f(w)^T=w^T. \]
This identity clearly shows that $f$ is injective. Moreover, $f$ is a proper holomorphic immersion (cf.\,Chan-Xiao-Yuan \cite{CXY17} and Mok \cite{Mok12}), and thus $f:D\to \Omega$ is a topological embedding.
Since $\overline{Jf({\bf 0})}^T$ is of rank $n$, the $n$-by-$N$ matrix
\[ M(f):=\begin{bmatrix}
{1\over \lambda_1}{\bf I}_{n_1} & & \\
& \ddots & \\
 & & {1\over \lambda_k} {\bf I}_{n_k}
 \end{bmatrix}
\overline{Jf({\bf 0})}^T
 \begin{bmatrix}
\mu_1{\bf I}_{N_1} & & \\
& \ddots & \\
 & & {\mu_m} {\bf I}_{N_m}
 \end{bmatrix} \]
is of rank $n$.
By the same constructions and arguments in the previous sections, we may define the map $\pi_f:\Omega \to \mathbb C^n$ by $\pi_f(z_1,\ldots,z_N)
:=(z_1,\ldots,z_N) M(f)^T$, and define
\[ D_f:=\{z\in \Omega:\pi_f(z)\in D \}.\]
By the Hermann convexity theorem (cf.\,Wolf \cite{Wo72}), $D$ and $\Omega$ are realized as bounded convex domains in complex Euclidean spaces.
Since $\pi_f$ is a complex linear map, $D_f\subset \Omega\subset \mathbb C^N$ is also a bounded convex domain.
Moreover, since $D$ and $\Omega$ are balanced domains, so is $D_f$ from the construction and the linearity of $\pi_f$. Hence, $D_f$ is a bounded balanced convex domain.

The restriction of $\pi_f$ to $D_f$ yields the surjective holomorphic submersion $\widetilde \pi_f:D_f \to D$ such that $S:=f(D)\subset D_f$ and $\widetilde\pi_f(f(w))=w$ for all $w\in D$.
Moreover, by the same arguments in the previous sections we have
\[ T_{f(w)}(\Omega) = T_{f(w)}(S)\oplus \ker d\widetilde \pi_f(f(w)) \quad\forall\;w\in D. \]
Note that $T_{D_f}=T_\Omega|_{D_f}$ and $\ker d\widetilde \pi_f \subset T_{D_f}$ is a holomorphic vector subbundle.
The above identity yields $T_\Omega|_S=T_S\oplus (\ker d\widetilde \pi_f)|_S$.

In general, without assuming $f({\bf 0})={\bf 0}$, we may replace $f$ by $\Psi\circ f$ in the above construction for some automorphism $\Psi\in {\rm Aut}(\Omega)$ of $\Omega$ such that $\Psi(f({\bf 0}))={\bf 0}$ since $\Omega$ is homogeneous under the action of its automorphism group. Note that $f:D\to \Omega$ is obviously still a proper holomorphic isometric embedding so that $f(D) \subset \Omega$ is a closed embedded complex submanifold.
In this case, we have a holomorphic submersion $\widetilde \pi_f:=\widetilde \pi_{\Psi\circ f}\circ \Psi: D_f:=\Psi^{-1}(D_{\Psi\circ f}) \to D$ such that $(\widetilde \pi_{\Psi\circ f}\circ \Psi) (f(w)) = w$ for all $w\in D$, and $f(D) \subset \Psi^{-1}(D_{\Psi\circ f})$.
Since $D_{\Psi\circ f}\subset \mathbb C^N$ is a bounded convex domain, $D_{\Psi\circ f}$ is a (bounded) domain of holomorphy in $\mathbb C^N$, and so is the open subset $\Psi^{-1}(D_{\Psi\circ f})\subset \Omega$. In particular, $\Psi^{-1}(D_{\Psi\circ f})$ is a Stein manifold.
Writing $S=f(D)$, we have $T_\Omega|_S=T_S\oplus (\ker d(\widetilde \pi_{\Psi\circ f}\circ\Psi))|_S$. It remains to show that this is a holomorphic splitting.

For simplicity, we write $\widetilde \pi_f:D_f \to D$ for the holomorphic submersion obtained above, and $f:D\to D_f\subset \Omega$ is a global holomorphic section of $\widetilde \pi_f$, i.e., $\widetilde \pi_f\circ f=\id_D$.
Moreover, we have the decomposition $T_\Omega|_S=T_S\oplus \ker(d\widetilde\pi_f)|_S$ in the sense that every $v\in T_x(\Omega)$, $x\in S$, can be written as $v = v_1+v_2$ for $v_1\in T_x(S)$, $v_2\in \ker(d\widetilde\pi_f(x))$, and $T_x(S)\cap \ker(d\widetilde\pi_f(x))=\{{\bf 0}\}$.
This yields the surjective holomorphic bundle map $\nu:T_\Omega|_S=T_{D_f}|_S \to \ker(d\widetilde\pi_f)|_S$ given by the canonical projection $v=v_1+v_2 \mapsto v_2\in \ker(d\widetilde\pi_f(x))$ for each $x\in S$.
Let $\iota_*: T_S \to T_\Omega|_S$ be the inclusion map $\iota_*(v)=v$ for all $v\in T_x(S)$, $x\in S$.
We have the short exact sequence
\begin{equation}\label{eq:sh_ex_sq}
0 \to T_S \xrightarrow{\iota_*} T_\Omega|_S=T_{D_f}|_S \xrightarrow{\nu} \ker(d\widetilde\pi_f)|_S \to 0
\end{equation}
of holomorphic vector bundles over $S$.
Since we have the holomorphic bundle isomorphism $df:T_{D}\to T_S$, for all $v\in T_x(S)$, $x=f(w)\in S$, $v=df(\alpha)$ for some $\alpha\in T_w(D)$, and
\[ \begin{split}
((df\circ d\widetilde \pi_f) \circ \iota_*)(v) 
=&(df\circ d\widetilde \pi_f)(v) = (df\circ d\widetilde \pi_f)(df(\alpha))\\
=& df (d(\widetilde \pi_f\circ f)(\alpha))
= df(\alpha)=v. 
\end{split}\]
This shows that $p_f:=df\circ d\widetilde \pi_f: T_{D_f}|_S = T_\Omega|_S \to T_S$ is a surjective holomporphic bundle map such that $p_f\circ \iota_*=\id_{T_S}$.
Hence, the above short exact sequence (\ref{eq:sh_ex_sq}) splits holomorphically and $T_\Omega|_S= T_S\oplus \ker(d\widetilde\pi_f)|_S$ is a holomorphic splitting (cf.\,Mok-Ng \cite[Section 4, p.\,1065]{MN17}).
By the same arguments, the tangent sequence
\begin{equation}\label{eq:ts}
0 \to T_S \xrightarrow{\iota_*} T_\Omega|_S \to N_{S|\Omega} \to 0
\end{equation}
also splits holomorphically over $S$. 
\end{proof}
\begin{remark}
In general, it follows from Abate-Bracci-Tovena {\rm\cite[Example 1.2, p.\,630]{ABT09}} that the tangent sequence splits holomorphically over a Stein {\rm(}embedded{\rm)} submanifold of a complex manifold, so that the tangent sequence {\rm(\ref{eq:ts})} splits holomorphically over $S=f(D)\subset \Omega$ by the Steinness of bounded symmetric domains.
However, our proof provides an explicit way to obtain this fact in the case of holomorphic isometries between bounded symmetric domains.
\end{remark}

As a consequence, we have the following corollary.

\begin{corollary}\label{cor_CI_III}
Let $f:(\mathbb B^n,g_{\mathbb B^n})\to (D^{\rm III}_m,g_{D^{\rm III}_m})$ be a holomorphic isometric embedding, where $m\ge 3$.
Write $S:=f(\mathbb B^n)$.
Then, the normal bundle $N_{S|D^{\rm III}_m}=T_{D^{\rm III}_m}|_S/T_S$ to $S$ in $D^{\rm III}_m$ is trivial, and $S$ is an ideal-theoretic complete intersection, i.e., there exist $\dim_{\mathbb C}(D^{\rm III}_m)-n$ holomorphic functions $g_j$, $1\le j\le \dim_{\mathbb C}(D^{\rm III}_m)-n$, on $D^{\rm III}_m$ that generate $\Gamma(D^{\rm III}_m,\mathcal I_S)$ over $\Gamma(D^{\rm III}_m,\mathcal O_{D^{\rm III}_m})$, where $\mathcal I_S$ denotes the ideal sheaf of $S$.
\end{corollary}
\begin{proof}
By Theorem \ref{thm_CI1}, the normal bundle $N_{S|D^{\rm III}_m}=T_{D^{\rm III}_m}|_S/T_S$ to $S$ in $D^{\rm III}_m$ is trivial.
In \cite{Mok16}, Mok proved that $n\le m$.
Then, we have $2(\dim_{\mathbb C}(D^{\rm III}_m)-1)-3n
\ge 2(\dim_{\mathbb C}(D^{\rm III}_m)-1)-3m
= 2(m(m+1)/2-1)-3m
=m(m-2)-2
\ge 3(3-2)-2 = 1>0$ so that
\[ n < {2(\dim_{\mathbb C}(D^{\rm III}_m)-1)\over 3}. \]
By B\v{a}nic\v{a}-Forster \cite[1.5.\,Corollary]{BF82}, $S$ is an ideal-theoretic complete intersection.
\end{proof}

Since the bounded symmetric domains $D$ and $\Omega$ are contractible Stein manifolds, we have the following generalization of Theorem \ref{thm_CI1} by using Theorem \ref{thm_CI_general}, Proposition \ref{pro:CI} and the proof of Theorem \ref{thm_CI1}.

\begin{theorem}\label{thm_CI2}
Let $D, \Omega$ and $S$ be as in Theorem {\rm\ref{thm_CI_general}}.
Then, the normal bundle $N_{S|\Omega}=T_\Omega|_S/T_S$ to $S$ in $\Omega$ is trivial, $S$ is a leaf in a nonsingular holomorphic foliation of $\Omega$, and $S$ is a set-theoretic complete intersection, i.e., there exist $N-n$ holomorphic functions $g_1,\ldots,g_{N-n}$ on $\Omega$ such that
\[ S=\{z\in \Omega: g_j(z)=0\;{\rm for}\;j=1,\ldots,N-n\}. \]
\end{theorem}

\subsection{Linear degeneracy of certain holomorphic isometries}
In \cite{Mok16}, Mok proved that if $F:(\mathbb B^n,g_{\mathbb B^n})\to (\Omega,g_\Omega)$ is a holomorphic isometry, where $\Omega\Subset \mathbb C^N$ is an irreducible bounded symmetric domain in its Harish-Chandra realization of rank $\ge 2$, then $n\le p(\Omega)+1$.
Let $f:(\mathbb B^{n},g_{\mathbb B^{n}})\to (\Omega,g_\Omega)$ be a holomorphic isometry.
Assume that $f({\bf 0})={\bf 0}$.
We have proved that
\begin{equation}
\begin{split}
f((V+{\bf v})\cap \mathbb B^{n})
&=(f_m^\sharp(V) + Jf({\bf 0}) {\bf v})\cap f(\mathbb B^{n})\\
&\subseteq
(f_m^\sharp(V) + Jf({\bf 0})  {\bf v})\cap \Omega
\end{split}
\end{equation}
for any ${\bf v}\in \mathbb B^n$, $V\in G(m,{n}-m)$, $1\le m\le n-1$, where $f_m^\sharp(V):=df_{\bf 0}(V) \oplus \ker\big(\overline{Jf({\bf 0})}^T\big)$.
It is natural to ask how could the preservation of complex affine-linear subspaces be used for the classification of such holomorphic isometries.
On the other hand, we have the following linear degeneracy of certain holomorphic isometries from $(\mathbb B^n,kg_{\mathbb B^n})$ to $(\Omega,g_\Omega)$ which map ${\bf 0}$ to ${\bf 0}$.

\begin{proposition}\label{Pro:Slicing=LS}
Let $F: (\mathbb B^n,kg_{\mathbb B^n})\to (\Omega,g_\Omega)$ be a holomorphic isometry such that $F({\bf 0})={\bf 0}$, where $\Omega\Subset \mathbb C^N$ is an irreducible bounded symmetric domain of rank $\ge 2$ in its Harish-Chandra realization and $k$ is an integer satisfying $1\le k\le \mathrm{rank}(\Omega)$.
If $F=f\circ \rho$ for some holomorphic isometries $f:(\mathbb B^{m},kg_{\mathbb B^{m}})\to (\Omega,g_{\Omega})$ and $\rho:(\mathbb B^n,g_{\mathbb B^n})\to (\mathbb B^{m},g_{\mathbb B^{m}})$ such that $m\ge n+1$, then $F$ is linearly degenerate in the sense that $F(\mathbb B^n) \subseteq W$ for some proper complex $(N-m+n)$-dimensional vector subspace $W$ of $\mathbb C^N$.
\end{proposition}
\begin{proof}
Under the assumption, we have $f(w_0)={\bf 0}$, where $w_0=\rho({\bf 0})$.
Note that there is $\Psi\in \Aut(\mathbb B^m)$ such that 
$\Psi({\bf 0})=w_0$.
Then, we have $F=(f\circ \Psi)\circ (\Psi^{-1}\circ \rho)$ so that $\widetilde f:= f\circ \Psi$ and $\widetilde\rho:=\Psi^{-1}\circ \rho$ satisfy $\widetilde f({\bf 0})={\bf 0}$ and $\widetilde\rho({\bf 0})={\bf 0}$.
Moreover, $\widetilde \rho:(\mathbb B^n,g_{\mathbb B^n})\to (\mathbb B^{m},g_{\mathbb B^{m}})$ is a totally geodesic holomorphic isometric embedding, i.e., $(\widetilde\rho(\mathbb B^n),g_{\mathbb B^{m}}|_{\widetilde\rho(\mathbb B^n)}) \subset (\mathbb B^{m},g_{\mathbb B^{m}})$ is totally geodesic.
Therefore, $\widetilde\rho(\mathbb B^n) = \mathbb B^m\cap V$ for some complex $n$-dimensional vector subspace $V\subset \mathbb C^m$ (cf.\,Mok \cite[Proposition 2.1]{Mok22}).
In other words, we have $F(\mathbb B^n) = \widetilde f(\mathbb B^m\cap V)$.
Then, by Proposition \ref{pro:Structure_Slicing} we have
\[ F(\mathbb B^n) = \widetilde f(\mathbb B^m\cap V) = \widetilde f(\mathbb B^{m})\cap W \subset W, \]
where $W:=d\widetilde f_{\bf 0}(V)\oplus \ker\overline{J\widetilde f({\bf 0})}^T$ is a complex $(N-m+n)$-dimensional vector subspace of $\mathbb C^N$, i.e., $F$ is linearly degenerate.
\end{proof}
\begin{remark}
By Proposition {\rm\ref{Pro:Slicing=LS}}, if one can find a holomorphic isometry $F: (\mathbb B^n,g_{\mathbb B^n})\to (\Omega,g_\Omega)$ such that $F({\bf 0})={\bf 0}$, $1\le n\le p(\Omega)$, and $F(\mathbb B^n)\subset \Omega \Subset \mathbb C^N$ is linearly non-degenerate, i.e., $F(\mathbb B^n)\not\subseteq W$ for any proper complex vector subspace $W$ of $\mathbb C^N$, then $F$ is not obtained from a slicing of $\mathbb B^{m}$ through $f$ for any holomorphic isometry $f:(\mathbb B^{m},g_{\mathbb B^{m}})\to (\Omega,g_{\Omega})$, where $n+1\le m \le p(\Omega)+1$.
This may require certain geometric properties of the target irreducible bounded symmetric domain $\Omega\Subset \mathbb C^N$ because the author and Mok {\rm\cite{CM17}} have proved that when $\Omega\cong D^{\mathrm{IV}}_N$ for some $N\ge 3$, any holomorphic isometry $F: (\mathbb B^n,g_{\mathbb B^n})\to (\Omega,g_\Omega)$, $1\le n\le p(\Omega)=N-2$, is obtained from a slicing of $\mathbb B^{p(\Omega)+1}$ through $f$ for some holomorphic isometry $f:(\mathbb B^{p(\Omega)+1},g_{\mathbb B^{p(\Omega)+1}})\to (\Omega,g_{\Omega})$.
In other words, any holomorphic isometry $F: (\mathbb B^n,g_{\mathbb B^n})\to (D^{\mathrm{IV}}_N,g_{D^{\mathrm{IV}}_N})$ with $F({\bf 0})={\bf 0}$ and $1\le n\le p(D^{\mathrm{IV}}_N)=N-2$, is linearly degenerate for $N\ge 3$.
\end{remark}

\begin{corollary}\label{cor:not-contain_fin_man_pts}
Let $\Omega\Subset \mathbb C^N$ be an irreducible bounded symmetric domain of rank $2$ in its Harish-Chandra realization such that $\Omega$ is not biholomorphic to $D^{\mathrm{I}}_{2,q}$ for any $q\ge 5$.
Let $X_c$ be the compact dual of $\Omega$ and $X_c\hookrightarrow \mathbb P\big(\Gamma(X_c,\mathcal O(1))^*\big)\cong \mathbb P^{N'}$ be the first canonical embedding.
Then, there exist $N'-N+3$ distinct points $x_1,\ldots,x_{N'-N+3}$ in $\Omega$ such that the image of any holomorphic isometry from $(\Delta,g_\Delta)$ to $(\Omega,g_\Omega)$ does not contain all of the points $x_1,\ldots,x_{N'-N+3}$.
\end{corollary}
\begin{proof}
Under the assumptions, it follows from Chan \cite{Ch18} that $2N-N'>1$ and any holomorphic isometry $f:(\Delta,g_\Delta)\to(\Omega,g_\Omega)$ can be factorized as
$f=F\circ \rho$ for some holomorphic isometries $F:(\mathbb B^n,g_{\mathbb B^n})\to (\Omega,g_\Omega)$ and $\rho:(\Delta,g_\Delta)\to (\mathbb B^n,g_{\mathbb B^n})$, where $n:=n_0(\Omega)=2N-N'\ge 2$.
By Proposition \ref{Pro:Slicing=LS}, if ${\bf 0}\in f(\Delta)$, then $f(\Delta)$ is linearly degenerate, i.e., $f(\Delta)\subset W$ for some $(N'-N+1)$-dimensional complex vector subspace $W\subsetneq \mathbb C^N$.
We choose distinct points $x_1,\ldots,x_{N'-N+3}\in \Omega$ such that $x_{N'-N+3}={\bf 0}$ and $x_1,\ldots,x_{N'-N+2}\in \Omega\subset \mathbb C^N \smallsetminus \{{\bf 0}\}$ are $\mathbb C$-linearly independent when they are viewed as vectors in $\mathbb C^N$.
This shows that any holomorphic isometry $f:(\Delta,g_\Delta)\to(\Omega,g_\Omega)$ does not contain all of the above chosen points $x_1,\ldots,x_{N'-N+3}$, otherwise we would have $f(\Delta)\subset W$ for some $(N'-N+1)$-dimensional complex vector subspace $W\subsetneq \mathbb C^N$ and $W$ would contain the $\mathbb C$-linearly independent vectors $x_1,\ldots,x_{N'-N+2}$, a plain contradiction.
\end{proof}
\begin{remark}
In {\rm\cite{Dr02}}, Drinovec Drnov\v{s}ek proved that given any discrete subset $S:=\{x_n\}_{n\in \mathbb N}$ of a connected Stein manifold $X$ with $\dim_{\mathbb C}X\ge 2$, there is a proper holomorphic immersion $f$ from the unit disk $\Delta$ to $X$ such that $S\subset f(\Delta)$.
Note that any bounded symmetric domain is a Stein manifold. 
In other words, the corollary shows that holomorphic isometries from the unit disk to any bounded symmetric domain $\Omega$ with respect to the Bergman metrics up to normalizing constants are much more rigid than proper holomorphic maps from the unit disk to $\Omega$ in general.
\end{remark}

We have the following consequence of Theorem \ref{thm_CI_general} regarding the linearly degeneracy of holomorphic isometries between bounded symmetric domains with respect to the canonical K\"ahler metrics.

\begin{proposition}\label{Pro:Fac=LD}
Let $f:(D,g'_D)\to (\Omega,g'_\Omega)$ be a holomorphic isometry such that $f({\bf 0})={\bf 0}$, where $D\Subset\mathbb C^n$ and $\Omega\Subset \mathbb C^N$ are bounded symmetric domains equipped with some canonical K\"ahler metrics $g'_D$ and $g'_\Omega$ respectively.
Suppose there are holomorphic isometries $F:(D,g'_D)\to (U,g'_U)$ and $G:(U,g'_U)\to (\Omega,g'_\Omega)$ such that $f=G\circ F$, $F({\bf 0})={\bf 0}$ and $F$ is linearly degenerate, i.e., $F(D)$ lies inside a complex $k$-dimensional vector subspace of $\mathbb C^m$ for some $k$, $1\le k<m:=\dim_{\mathbb C}(U)$, where $U \Subset \mathbb C^m$ is a bounded symmetric domain equipped with some canonical K\"ahler metric $g'_U$.
Then, $f$ is linearly degenerate, i.e., $f(D) \subseteq W$ for some proper complex vector subspace $W$ of $\mathbb C^N$ with $\dim_{\mathbb C}(W)=N-m+k$.
\end{proposition}
\begin{proof}
Write $G=(G_1,\ldots,G_N)$, $F=(F_1,\ldots,F_m)$ and $f=(f_1,\ldots,f_N)$.
From the assumptions we have $G({\bf 0})={\bf 0}$.
By Theorem \ref{thm_CI_general}, there is a $m$-by-$N$ matrix $M(G)$ of rank $m$ such that
\[ (G_1(\zeta),\ldots,G_N(\zeta)) M(G)^T = (\zeta_1,\ldots,\zeta_m) \]
for all $(\zeta_1,\ldots,\zeta_m)\in U\Subset \mathbb C^m$.
Then, we have 
\[\begin{split}
 (f_1(w),\ldots,f_N(w))M(G)^T &=(G_1(F(w)),\ldots,G_N(F(w)) ) M(G)^T \\
 &= (F_1(w),\ldots,F_m(w)) 
\end{split}\]
for all $w\in D\Subset \mathbb C^n$.
Since $F$ is linearly degenerate, $F(D)$ lies inside a complex $k$-dimensional vector subspace of $\mathbb C^m$, for some $k$, $1\le k<m$, i.e., there is a $m$-by-$(m-k)$ matrix $A$ of rank $m-k$ such that $(F_1(w),\ldots,F_m(w))A\equiv 0$.
Then, $M(G)^TA$ is an $N$-by-$(m-k)$ matrix such that
\[ \begin{split}
{\rm rank}(M(G)^TA) &\ge {\rm rank}(M(G)^T)+{\rm rank}(A) - m \\
&= m+(m-k)-m = m-k > 0, 
\end{split}\]
i.e., $M(G)^TA$ is a non-zero rank-$(m-k)$ matrix, and
\[ (f_1(w),\ldots,f_N(w)) M(G)^TA = (F_1(w),\ldots,F_m(w))A \equiv 0. \]
This implies that $f$ is also linearly degenerate since its image lies inside the proper complex vector subspace
\[ W:=\{(z_1,\ldots,z_N)\in \mathbb C^N: (z_1,\ldots,z_N)M(G)^TA = 0 \} \]
of $\mathbb C^N$.
In addition, $\dim_{\mathbb C}(W)=N-(m-k)$.
\end{proof}

\subsection{Non-degeneracy of holomorphic isometries}

\subsubsection{Sufficiently non-degeneracy of the system of functional equations}
Recall the following definition made by Mok \cite{Mok12}.

\begin{defn}[cf.\,Mok \cite{Mok12}]\label{defn:suff_non-degen}
Let $D\Subset \mathbb C^n$ and $\Omega\Subset \mathbb C^N$ be bounded complete circular domains and let $F: (D,\lambda ds_D^2)\to (\Omega,ds_\Omega^2)$ be a holomorphic isometry such that $F({\bf 0})={\bf 0}$, where $\lambda>0$ is some real constant.
Then, the system of functional equations of $F$ is said to be \textbf{sufficiently non-degenerate} if and only if any irreducible component of the complex-analytic subvariety
\[ V_F:=\bigcap_{\zeta\in D} \left\{(w,z)\in D\times \Omega : K_\Omega(z,F(\zeta)) - A K_D(w,\zeta)^\lambda = 0 \right\} \]
containing the graph of $F$ is of complex dimension $n$, where $A:={K_\Omega({\bf 0},{\bf 0})\over  K_D({\bf 0},{\bf 0})^\lambda}>0$ is a real constant.
\end{defn}
\begin{remark}
Recall that a domain $U \subset \mathbb C^n$ is complete if and only if ${\bf 0}\in U$, and a circular domain $U'\subset \mathbb C^n$ is a domain such that $e^{i\theta}z\in U'$ for all $z\in U'$ and all $\theta\in \mathbb R$.
For example, every bounded symmetric domain is a bounded complete circular domain in its Harish-Chandra realization.
In general, we may consider a germ $F: (D,\lambda ds_D^2;{\bf 0})\to (\Omega,ds_\Omega^2;{\bf 0})$ of holomorphic isometry, and
\[ V_F:=\bigcap_{\zeta\in \mathbb B^n({\bf 0},\varepsilon)}\{(w,z)\in D\times \Omega: K_\Omega(z,F(\zeta))
- A K_D(w,\zeta)^\lambda = 0\} \]
for some sufficiently small $\varepsilon>0$.
Then, we could make the same definition of sufficiently non-degeneracy of the system of functional equations of $F$.
\end{remark}

Let $F$ be as in Definition \ref{defn:suff_non-degen}. One may ask if the sufficiently non-degeneracy of the system of functional equations of $F$ is preserved under the transformations of $F$ by the automorphisms of $D$ and $\Omega$. We let $\Upsilon\in \Aut(\Omega)$ and $\varphi\in \Aut(D)$ be such that $(\Upsilon\circ (F\circ \varphi))({\bf 0})={\bf 0}$.
Then, we also have the system of functional equations
\[ K_\Omega(z,(\Upsilon\circ (F\circ \varphi))(\zeta))
= A K_D(w,\zeta)^\lambda \]
for all $\zeta\in D$, whose common zero set in $D\times \Omega$ is denoted by $V_{\Upsilon\circ (F\circ \varphi)}$.
The question is whether the system of functional equations of $F$ is sufficiently non-degenerate if and only if the system of functional equations of $\Upsilon\circ (F\circ \varphi)$ is sufficiently non-degenerate. This can be done as follows.

Write $F_{\Upsilon,\varphi}:=\Upsilon\circ (F\circ \varphi)$ for simplicity.
Define $\hat\Phi \in \Aut(D\times \Omega)$ by $\hat\Phi(w,z):=(\varphi^{-1}(w),\Upsilon(z))$ for $(w,z)\in D\times \Omega$.
We have $\hat\Phi({\rm Graph}(F))={\rm Graph}(F_{\Upsilon,\varphi})$ and
\[ \begin{split}
\hat\Phi(V_F)
=&\{(\varphi^{-1}(w),\Upsilon(z))\in D\times \Omega: K_\Omega(z,F(\zeta))
= A K_D(w,\zeta)^\lambda\;\;\forall\;\zeta\in D\}\\
=&\{(w',z')\in D\times \Omega: 
K_\Omega(\Upsilon^{-1}(z'),F(\zeta))
= A K_D(\varphi(w'),\zeta)^\lambda\;\;\forall\;\zeta\in D\}
\end{split} \]
by putting $z':=\Upsilon(z)$ and $w'=\varphi^{-1}(w)$.
We may rewrite $K_\Omega(\Upsilon^{-1}(z'),F(\zeta))
= A K_D(\varphi(w'),\zeta)^\lambda$ as
\[ \begin{split}
&K_\Omega(z',F_{\Upsilon,\varphi}(\zeta')) \det J\Upsilon(\Upsilon^{-1}(z')) \cdot \overline{\det J\Upsilon(F(\varphi(\zeta')))}\\
=& A K_D(w',\zeta')^\lambda 
\left(\det J\varphi(w') \cdot \overline{\det J\varphi(\zeta')} \right)^\lambda,
\end{split}\]
where $\zeta':=\varphi^{-1}(\zeta)$.
Putting $z'=F_{\Upsilon,\varphi}(w')=\Upsilon(F(\varphi(w')))$ in the above equation we have
\[ \det J\Upsilon(F (\varphi(w')) ) \cdot \overline{\det J\Upsilon(F(\varphi(\zeta')))}
=\left(\det J\varphi(w') \cdot \overline{\det J\varphi(\zeta')} \right)^\lambda \]
for all $w',\zeta'\in D$ since $K_\Omega(F_{\Upsilon,\varphi}(w'),F_{\Upsilon,\varphi}(\zeta'))
=A K_D(w',\zeta')^\lambda$.
Hence, the system of equations becomes
\begin{equation} \label{Eq:Sys_FE_New_Trans1}
K_\Omega(z',F_{\Upsilon,\varphi}(\zeta')) \det J\Upsilon(\Upsilon^{-1}(z'))
= A K_D(w',\zeta')^\lambda 
\det J\Upsilon(F (\varphi(w')) ). 
\end{equation}
for all $\zeta'\in D$.
Putting $\zeta'={\bf 0}$, we have
\[ \det J\Upsilon(\Upsilon^{-1}(z'))=\det J\Upsilon(F (\varphi(w')) )\]
for all solutions $(w',z')\in D\times \Omega$ of the system of equations (\ref{Eq:Sys_FE_New_Trans1}) since $A={K_\Omega({\bf 0},{\bf 0})\over K_D({\bf 0},{\bf 0})^\lambda}$, and $K_\Omega(z,{\bf 0})$, $K_D(w,{\bf 0})$ are positive constants by Mok \cite[Section 1.1, p.\,1620]{Mok12}.
Thus, we have
\[ \hat\Phi(V_F) \subseteq \{(w',z')\in D\times \Omega: K_\Omega(z',F_{\Upsilon,\varphi}(\zeta')) 
= A K_D(w',\zeta')^\lambda\;\;\forall\;\zeta'\in D\},\]
i.e., $\hat\Phi(V_F)\subseteq V_{\Upsilon\circ (F\circ \varphi)}$.
Conversely, we consider the inverse map $\hat\Phi^{-1}\in {\rm Aut}(D\times \Omega)$ given by $\hat\Phi^{-1}(w',z'):=(\varphi(w'),\Upsilon^{-1}(z'))$ for $(w',z')\in D\times \Omega$. Then, by similar computations we have $\hat\Phi^{-1}(V_{\Upsilon\circ (F\circ \varphi)})\subseteq V_{F}$ so that $V_{\Upsilon\circ (F\circ \varphi)} \subseteq \hat\Phi(V_{F})$. Therefore, $\hat\Phi(V_F)= V_{\Upsilon\circ (F\circ \varphi)}$.
Now, for any $w\in D$ we have
\[ \begin{split}
\dim_{(w,F(w))}(V_F)
=& \dim_{(\varphi^{-1}(w),\Upsilon(F(w)))}
\left(\hat\Phi(V_F)\right)\\
=&\dim_{(\varphi^{-1}(w),\Upsilon(F(w)))}
(V_{\Upsilon\circ (F\circ \varphi)})\\
=& \dim_{(w',F_{\Upsilon,\varphi}(w'))}
(V_{F_{\Upsilon,\varphi}}),
\end{split}\]
where $w'=\varphi^{-1}(w)\in D$.
Note that $\dim_{\mathbb C}({\rm Graph}(F))$ $=$ $\dim_{\mathbb C}({\rm Graph}(F_{\Upsilon,\varphi}))$ $=$ $n$.
Hence, any irreducible component of $V_F$ containing ${\rm Graph}(F)$ is of complex dimension $n$ if and only if any irreducible component of $V_{F_{\Upsilon,\varphi}}$ containing ${\rm Graph}(F_{\Upsilon,\varphi})$ is of complex dimension $n$.
To summarize, we have obtained the following proposition.

\begin{proposition}\label{Pro:Sys_FE_suff_ND_Well-def_Equiv_class}
Let $D\Subset \mathbb C^n$ and $\Omega\Subset \mathbb C^N$ be bounded complete circular domains and let $F: (D,\lambda ds_D^2)\to (\Omega,ds_\Omega^2)$ be a holomorphic isometry such that $F({\bf 0})={\bf 0}$, where $\lambda>0$ is some real constant.
Let $\Upsilon\in \Aut(\Omega)$ and $\varphi\in \Aut(D)$ be such that $(\Upsilon\circ (F\circ \varphi))({\bf 0})={\bf 0}$.
Then, the system of functional equations of $F$ is sufficiently non-degenerate if and only if the system of functional equations of $\Upsilon\circ (F\circ \varphi)$ is sufficiently non-degenerate.
\end{proposition}
\begin{remark}
As a special case, Proposition {\rm\ref{Pro:Sys_FE_suff_ND_Well-def_Equiv_class}} holds when $D$ and $\Omega$ are bounded symmetric domains in their Harish-Chandra realizations.
More generally, we may consider a germ $F: (D,\lambda ds_D^2;{\bf 0})\to (\Omega,ds_\Omega^2;{\bf 0})$ of holomorphic isometry.
Then, Proposition {\rm\ref{Pro:Sys_FE_suff_ND_Well-def_Equiv_class}} still holds under the following extra assumption.
\begin{enumerate}
\item[($\star$)] Denoting $D_0\subset D$ the domain of the germ $F$, $\varphi\in \Aut(D)$ is chosen so that ${\bf 0}\in \varphi^{-1}(D_0)$.
\end{enumerate}
This is only to make sure that we have the well-defined germ $\Upsilon\circ (F\circ \varphi): (D,\lambda ds_D^2;{\bf 0})$ $\to$ $(\Omega,ds_\Omega^2;{\bf 0})$ of holomorphic isometry, and the domain of $\Upsilon\circ (F\circ \varphi)$ is $\varphi^{-1}(D_0)$.
\end{remark}

\subsubsection{Linearly non-degeneracy}
Let $D\Subset \mathbb C^n$ and $\Omega\Subset \mathbb C^N$ be bounded complete circular domains and let $F: (D,\lambda ds_D^2;{\bf 0})\to (\Omega,ds_\Omega^2;{\bf 0})$ be a germ of holomorphic isometry, where $\lambda>0$ is a real constant.
Suppose the system of functional equations of $F$ is not sufficiently non-degenerate.
In \cite{Mok12}, Mok has proved that there exists a family $\{h_\alpha\}_{\alpha\in {\bf A} }$ of extremal functions on $\Omega$ such that $F(\mathbb B^n({\bf 0},\varepsilon))\subset \bigcap_{\alpha\in {\bf A}} {\rm{Zero}}(h_\alpha)$, where $\varepsilon>0$ is sufficiently small such that $\mathbb B^n({\bf 0},\varepsilon)$ is contained in the domain of $F$.
Denoting by $H^2(\Omega)$ the space of square-integrable holomorphic functions on $\Omega \Subset \mathbb C^N$ with respect to the Lebesgue measure on $\mathbb C^N$, $h_\eta\in H^2(\Omega)$ is called an \emph{extremal function} adapted to a non-zero holomorphic tangent vector $\eta \in T_{x}(\Omega)$, $x\in \Omega$, whenever the maximum of $|dh(\eta)|$ is attained at $h=h_\eta$ among all $h\in H^2(\Omega)$ of unit norm satisfying $h(x)=0$. We refer the readers to Mok \cite[pp.\,1626--1627]{Mok12} for details.
When $D$ and $\Omega$ are bounded symmetric domains such that $\Omega$ is irreducible, in Chan-Mok \cite{CM17} we have shown that $ {\rm{Zero}}(h_\alpha)= {\rm{Zero}}(h'_\alpha|_\Omega)$ for some holomorphic function $h'_\alpha$ on $\mathbb C^N$ that is a $\mathbb C$-linear combination of $H_1,\ldots,H_{N'}$, where $H_l$'s are in the expression $h_\Omega(z,\xi)=1+\sum_{l=1}^{N'} H_l(z)\overline{H_l(-\xi)}$ with $H_j(z)=z_j$, $1\le j\le N$, and $H_l(z):=G_{l-N}(z)$, $N+1\le l\le N'$, in Equation {\rm(}\ref{Eq:sp_form_poly_h}{\rm)}.
Note that by Mok \cite{Mok12}, $F$ extends to a global holomorphic isometry from $(D,\lambda ds_D^2)$ to $(\Omega,ds_\Omega^2)$, which is still denoted by $F$ for simplicity.

It's not clear if $h'_\alpha$ is actually a linear function on $\mathbb C^N$ for some $\alpha\in {\bf A}$.
Therefore, it is natural to ask if such a holomorphic isometry $F$ is actually affine-linearly degenerate, i.e., $F(D)\subset W$ for some proper affine-linear subspace $W\subsetneq \mathbb C^N$.
Indeed, we have the affirmative answer to this question.
More precisely, we will prove that for any holomorphic isometry $F$ between bounded symmetric domains with respect to the Bergman metrics up to a scalar constant, linearly non-degeneracy of the image of $F$ implies the system of functional equations of $F$ is sufficiently non-degenerate.
We first deal with the case of irreducible bounded symmetric domains, as follows.

\begin{proposition}\label{Pro:Linear_Nondegen1}
Let $F: (D,\lambda g_{D})\to (\Omega,g_\Omega)$ be a holomorphic isometry such that $F({\bf 0})={\bf 0}$, where $D\Subset \mathbb C^n$ and $\Omega\Subset \mathbb C^N$ are irreducible bounded symmetric domains and $\lambda>0$ is a real constant.
If $F$ is affine-linearly non-degenerate in the sense that $F(D)\not\subset W$ for any proper complex affine-linear subspace $W\subsetneq \mathbb C^N$, then the system of functional equations of $F$ is sufficiently non-degenerate.
\end{proposition}
\begin{proof}
From Chan-Mok \cite{CM17}, we have
\[ h_\Omega(F(w),F(\zeta)) = h_D(w,\zeta)^\lambda \]
and $h_\Omega(z,\xi) = 1-\sum_{j=1}^N z_j\overline{\xi_j} + \sum_{l=N+1}^{N'}(-1)^{\deg G_l} G_l(z)\overline{G_l(\xi)}$, where $G_l(z)$, $N+1\le l\le N'$, are homogeneous polynomials in $z$ of degree $\ge 2$.
Write $F=(F^1,\ldots,F^N)$ and let
$\Psi_\zeta(w,z):= 
h_\Omega(z,F(\zeta))-h_D(w,\zeta)^\lambda$ for any $\zeta\in D$.
Let $V\subset D\times \Omega$ be a complex-analytic subvariety defined by
\begin{equation}\label{Eq:Vari_Sys_FE1}
V:=\bigcap_{\zeta\in D} \{ (w,z)\in D\times \Omega: \Psi_\zeta(w,z)=0 \}.
\end{equation}
In Chan-Mok \cite{CM17}, we observed that
\begin{equation}\label{Eq:Vari_Sys_FE2}
 V=\left\{(w,z)\in D\times \Omega:
{\partial^{|I|}\over \partial \overline\zeta^I} \Psi_\zeta(w,z)\big|_{\zeta={\bf 0}} = 0,\;\forall \;I,\;|I|\ge 1\right\}.
\end{equation}
Note that
\[\begin{split}
&{\partial^{|I|}\over \partial \overline\zeta^I} \Psi_\zeta(w,z)\bigg|_{\zeta={\bf 0}}\\
=& 
-\sum_{j=1}^N \overline{{\partial^{|I|}F^j\over \partial \zeta^I}({\bf 0})}  z_j + \sum_{l=N+1}^{N'}(-1)^{\deg G_l} {\partial^{|I|}\over \partial \overline\zeta^I}\overline{G_l(F(\zeta))}\bigg|_{\zeta={\bf 0}}G_l(z)\\
&- {\partial^{|I|}\over \partial \overline\zeta^I} h_D(w,\zeta)^\lambda\bigg|_{\zeta={\bf 0}}
\end{split}\]
for multi-indices $I$ with $|I|\ge 1$.
Since $F$ is linearly non-degenerate, there exist distinct multi-indices $I_1,\ldots,I_N$ such that the complex linear span of ${\partial^{|I_\mu|}F\over \partial \zeta^{I_\mu}}({\bf 0})$, $1\le \mu\le N$, is the whole $\mathbb C^N$ if we view $F(w)=(F^1(w),\ldots,F^N(w))$ as a vector in $\mathbb C^N$ for each $w\in D$.
Let $V_1\subset D\times \Omega$ be a complex-analytic subvariety defined by the equations
${\partial^{|I_\mu|}\over \partial \overline\zeta^{I_\mu}} \Psi_\zeta(w,z)\bigg|_{\zeta={\bf 0}} = 0$
for $1\le \mu\le N$.
Write $\Phi_\mu(w,z):={\partial^{|I_\mu|}\over \partial \overline\zeta^{I_\mu}} \Psi_\zeta(w,z)\big|_{\zeta={\bf 0}}$, $\Phi:=(\Phi_1,\ldots,\Phi_N)$, $J_z \Phi 
:=\begin{pmatrix}{\partial\over \partial z_j}\Phi_i(w,z) \end{pmatrix}_{1\le i \le N,\;1\le j\le N}$ and $J_w \Phi$ $:=$ $\begin{pmatrix}{\partial\over \partial w_j}\Phi_i(w,z)\end{pmatrix}_{1\le i \le N,\;1\le j\le n}$. Then, the Jacobian matrix of $V_1$ at $({\bf 0},{\bf 0})$ is the $N\times (n+N)$ matrix
$M:= J\Phi ({\bf 0},{\bf 0}) = \begin{bmatrix} J_w \Phi ({\bf 0},{\bf 0}) & J_z \Phi ({\bf 0},{\bf 0})
\end{bmatrix}$, which is of full rank $N$ because
\[ J_z \Phi ({\bf 0},{\bf 0})=\begin{pmatrix}
-\overline{{\partial^{|I_l|}F^k\over \partial \zeta^{I_l}}({\bf 0})}
\end{pmatrix}_{1\le l,k \le N} \]
is of full rank $N$ from the above settings.
Therefore, $V_1$ is smooth and of dimension $n$ at $({\bf 0},{\bf 0})$.
Since $\mathrm{Graph}(F)\subset V_1$, $\mathrm{Graph}(F)$ is the irreducible component of $V_1$ containing $({\bf 0},{\bf 0})$.
On the other hand, we have $V\subseteq V_1$ so that any irreducible component of $V$ containing $\mathrm{Graph}(F)$ is contained in some irreducible component of $V_1$.
Hence, $\mathrm{Graph}(F)$ is the irreducible component of $V$ containing $\mathrm{Graph}(F)$, which is of dimension $n$.
In particular, the system of functional equations of $F$ is sufficiently non-degenerate in the sense of Mok \cite{Mok12}.
\end{proof}

One can easily obtain the following generalization of Proposition \ref{Pro:Linear_Nondegen1}.

\begin{proposition}\label{Pro:Linear_Nondegen2}
Let $F: (D,\lambda ds_D^2)\to (\Omega,ds_\Omega^2)$ be a holomorphic isometry such that $F({\bf 0})={\bf 0}$, where $D\Subset \mathbb C^n$ and $\Omega\Subset \mathbb C^N$ are bounded symmetric domains and $\lambda>0$ is a real constant.
Then, the assertion of Proposition {\rm\ref{Pro:Linear_Nondegen1}} still holds.
\end{proposition}
\begin{proof}
In this case, we consider the system of functional equations
\[ K_\Omega(z,F(\zeta)) = A K_D(w,\zeta)^\lambda \]
for $\zeta\in D$, where $A:={K_\Omega({\bf 0},{\bf 0})\over  K_D({\bf 0},{\bf 0})^\lambda}>0$ is a real constant (cf.\,Mok \cite{Mok12}).
Note that for any bounded symmetric domain $U\Subset \mathbb C^N$ the Bergman kernel of $U$ is given by
\[ K_U(z,\xi) = {C_U \over Q_U(z,\xi)} \]
for $z,\xi\in U$, where $Q_U(z,\xi)$ is a polynomial in $(z,\overline{\xi})$ such that $Q_U({\bf 0},\xi)\equiv 1$ and $C_U:=K_U({\bf 0},{\bf 0})$ is a positive real constant depending on $U$ (cf.\,Mok \cite{Mok12}).
Actually, $C_U={1\over {\rm Vol}_{\rm{Euc}}(U)}$ where ${\rm Vol}_{\rm{Euc}}(U):=\int_U dV$ is the Euclidean volume of $U$ and $dV$ denotes the Lebesgue measure (cf.\,Mok \cite[Proposition 1, p.\,56]{Mok89}).
Now, we write $\Omega=\Omega_1\times \cdots \times \Omega_m$ as a product of irreducible bounded symmetric domains $\Omega_j\Subset \mathbb C^{N_j}$, $1\le j\le m$, so that
$\sum_{j=1}^m N_j = N$.
Write $z=(z^1,\ldots,z^m)$, $\xi=(\xi^1,\ldots,\xi^m) \in \Omega$, where $z^j=(z^j_1,\ldots,z^j_{N_j})\in \Omega_j$ and $\xi^j=(\xi^j_1,\ldots,\xi^j_{N_j})\in \Omega_j$ for $j=1,\ldots,m$.
Then, we have
\[ Q_\Omega(z,\xi) = \prod_{j=1}^m h_{\Omega_j}(z^j,\xi^j)^{p_j}
=1-\sum_{j=1}^m \left( p_j \sum_{\mu=1}^{N_j} z^j_{\mu} \overline{\xi^j_{\mu}} \right) + Q'_\Omega(z,\xi) \]
for some positive integers $p_j$ depending only on $\Omega_j$, $1\le j\le m$, where $Q'_\Omega(z,\xi)$ is a polynomial in $(z,\overline{\xi})$ consisting of higher order terms so that ${\partial \over \partial z^j_\mu} Q'_\Omega(z,\xi)|_{z={\bf 0}} = 0$ for $1\le \mu \le N_j$, $1\le j\le m$.
The system of functional equations becomes
\[ \Psi_\zeta(w,z):=  Q_\Omega(z,F(\zeta)) -  Q_D(w,\zeta)^\lambda  = 0 \]
for $\zeta\in D$. Then, we may define $V$ as in (\ref{Eq:Vari_Sys_FE1}) and (\ref{Eq:Vari_Sys_FE2}).

Write $F=(F_{(1)},\ldots,F_{(m)})$, where each $F_{(j)}=(F_{(j)}^1,\ldots,F_{(j)}^{N_j})$ is a holomorphic map from $D$ to $\Omega_j$, $1\le j\le m$.
Similar to the proof of Proposition \ref{Pro:Linear_Nondegen1}, we may choose multi-indices $I_1,\ldots,I_N$ such that the $N\times N$ matrix $\begin{pmatrix}
-\overline{{\partial^{|I_l|}F\over \partial \zeta^{I_l}}({\bf 0})}
\end{pmatrix}_{1\le l \le N}$ is of rank $N$.
We may define $\Phi_\mu,\Phi,J_z\Phi$ and $J_w\Phi$ as in the proof of Proposition \ref{Pro:Linear_Nondegen1}.
By similar computations, we have
\[ J_z\Phi({\bf 0},{\bf 0})
= \begin{pmatrix} -p_1 \overline{{\partial^{|I_l|}F_{(1)}\over \partial \zeta^{I_l}}({\bf 0})},\ldots,-p_m \overline{{\partial^{|I_l|}F_{(m)}\over \partial \zeta^{I_l}}({\bf 0})} \end{pmatrix}_{1\le l\le N}. \]
Note that if $v_1,\ldots, v_N \in M(N,1;\mathbb C)$ are column vectors such that the $N\times N$ matrix $\begin{bmatrix}
v_1,\ldots,v_N
\end{bmatrix}$ is of rank $N$, equivalently, $v_1,\ldots,v_N$ are $\mathbb C$-linearly independent, then it is obvious that $\mu_1 v_1,\ldots,\mu_N v_N$ are $\mathbb C$-linearly independent for any non-zero constants $\mu_1,\ldots,\mu_N$, and thus $\begin{bmatrix}
\mu_1v_1,\ldots,\mu_N v_N
\end{bmatrix}$ is also of rank $N$.
It follows that $J_z\Phi({\bf 0},{\bf 0})$ is of rank $N$. The result follows from the same arguments in the proof of Proposition \ref{Pro:Linear_Nondegen1}.
\end{proof}
\begin{remark}
More generally, we may consider the canonical K\"ahler metrics $g'_D$ and $g'_\Omega$ on $D$ and $\Omega$ respectively instead of the Bergman metrics so that $(D,g'_D)\cong (D_1,\lambda_1 g_{D_1})\times\cdots \times (D_k,\lambda_k g_{D_k})$ and $(\Omega,g'_\Omega)\cong (\Omega_1,\mu_1 g_{\Omega_1})\times \cdots \times (\Omega_m, \mu_m g_{\Omega_m})$, where $D_1,\ldots,D_k$ {\rm(}resp.\,$\Omega_1,\ldots,\Omega_m${\rm)} are the irreducible factors of $D$ {\rm(}resp.\,$\Omega${\rm)}, $\lambda_j>0$, $1\le j\le k$, and $\mu_i>0$, $1\le i\le m$, are real constants.
We let $F=(F_{(1)},\ldots,F_{(m)}):(D,g'_D;{\bf 0})\to (\Omega,g'_\Omega;{\bf 0})$ be a germ of holomorphic isometry, where $F_{(j)}:(D;{\bf 0})\to (\Omega_j;{\bf 0})$ is a germ of holomorphic map for $1\le j\le m$. Then, we have the system of functional equations
\begin{equation}\label{Eq:Sys_FE_Gen}
\prod_{j=1}^m h_{\Omega_j}(z^j,F_{(j)}(\zeta))^{\mu_j}
=\prod_{i=1}^k h_{D_i}(w^i,\zeta^i)^{\lambda_i}
\end{equation}
for $\zeta=(\zeta^1,\ldots,\zeta^k)\in \mathbb B^n({\bf 0},\varepsilon)
\subset D_1\times\cdots\times D_k = D$, where $\varepsilon>0$ is a sufficiently small number. The same proof of Proposition {\rm\ref{Pro:Linear_Nondegen2}} shows that if $F$ is linearly non-degenerate, then the system of functional equations {\rm(\ref{Eq:Sys_FE_Gen})} of $F$ is sufficiently non-degenerate, meaning that any irreducible component of the complex-analytic subvariety $V\subset D\times \Omega$ defined by the system of functional equations {\rm(\ref{Eq:Sys_FE_Gen})} containing the graph of $F$ is of complex dimension $n$.
\end{remark}

In general, the converse of Proposition \ref{Pro:Linear_Nondegen2} does not hold so that for holomorphic isometries between bounded symmetric domains, linearly non-degeneracy is strictly stronger than sufficiently non-degeneracy of the system of functional equations.
More precisely, following the notations in the above remark, even if $F$ is affine-linearly degenerate, i.e., $F(D) \subset W$ for some proper affine-linear subspace $W\subsetneq \mathbb C^N$, the system of functional equations {\rm(}\ref{Eq:Sys_FE_Gen}{\rm)} could still be sufficiently non-degenerate as in the following example.

\begin{Example}
Let $F:(\Delta,2ds_\Delta^2)\to (\Delta^2,ds_{\Delta^2}^2)$ be the holomorphic isometry defined by $F(w):=(w,w)$ for $w\in \Delta$.
Then, $F(\Delta) \subset \{(z_1,z_2)\in \mathbb C^2: z_1-z_2=0\}$ so that $F$ is linearly degenerate.
Let $V\subset \Delta \times \Delta^2$ be the complex-analytic subvariety defined by the system of functional equations of $F$.
By computations and {\rm(\ref{Eq:Vari_Sys_FE2})} we have
\[ V = \left\{(w,z_1,z_2)\in \Delta^3: -z_1-z_2 + 2w=0,\; 2z_1z_2-2w^2 = 0 \right\}. \]
The equation $-z_1-z_2 + 2w=0$ can be rewritten as $w={1\over 2}(z_1+z_2)$.
Substituting $w={1\over 2}(z_1+z_2)$ into the equation $2z_1z_2-2w^2 = 0$, we have the equation $-{1\over 4}(z_1-z_2)^2=0$ whose zero set in $\Delta^3$ is $\{(w,z_1,z_2)\in \Delta^3: z_1-z_2=0\}$.
Therefore, $V$ can be identified as
\[ V=
\left\{(w,z_1,z_2)\in \Delta^3: w={1\over 2}(z_1+z_2),\; z_1-z_2 = 0\right\}, \]
which is obviously of complex dimension $1$.
Alternatively, one can check that $V\cap (\{w\} \times \Delta^2) = \{(w,w,w)\}$ so that $\dim_{(w,F(w))} (V\cap (\{w\} \times \Delta^2)) = 0$ for all $w\in \Delta$.
We compute
\[ \begin{split}
0=&\dim_{(w,F(w))} (V\cap (\{w\} \times \Delta^2))\\ \ge& \dim_{(w,F(w))}(V) + \dim_{(w,F(w))}(\{w\} \times \Delta^2) - \dim_{(w,F(w))}(\Delta^3)\\
=& \dim_{(w,F(w))}(V)-1
\end{split}\]
so that
\[ 1\ge \dim_{(w,F(w))}(V) \ge \dim_{(w,F(w))}({\rm{Graph}}(F)) = 1,\]
i.e., $\dim_{(w,F(w))}(V)=1$.
This implies that any irreducible component of $V$ containing ${\rm{Graph}}(F)$ is of complex dimension $1$.
Hence, the system of functional equations of $F$ is sufficiently non-degenerate.
\end{Example}

\paragraph{\textbf{Acknowledgements}}
The author is partially supported by a program of the Chinese Academy of Sciences.
We would also like to thank the referee for helpful comments.


\begin{thebibliography}{XXXXX}
\bibitem[ABT09]{ABT09} Abate, M., Bracci, F., Tovena, F.: \emph{Embeddings of submanifolds and normal bundles}, Adv. Math. {\bf 220} (2009), 620--656.

\bibitem[BF82]{BF82} B\v{a}nic\v{a}, C., Forster, O.: \emph{Complete intersections in Stein manifolds}, Manuscripta Math. {\bf 37} (1982), 343--356.

\bibitem[Bo78]{Bo78} Boraty\'nski, M.:
\emph{A note on set-theoretic complete intersection ideals}, J. Algebra {\bf 54} (1978), 1--5.

\bibitem[Ca53]{Ca53} Calabi, E.:
\emph{Isometric imbedding of complex manifolds},
Ann. of Math. {\bf 58} (1953), pp. 1--23.

\bibitem[Ch16]{Ch16} Chan, S. T.: \emph{On holomorphic isometric embeddings of complex unit balls into bounded symmetric domains}, Ph.D. thesis at The University of Hong Kong, 2016.

\bibitem[Ch18]{Ch18} Chan, S. T.: \emph{On the structure of holomorphic isometric embeddings of complex unit balls into bounded symmetric domains}, Pacific Journal of Mathematics, Vol. {\bf 295} (2018), 291--315; DOI: 10.2140/pjm.2018.295.291.

\bibitem[CM17]{CM17} Chan, S. T., Mok, N.: \emph{Holomorphic isometries of $\mathbb B^m$ into bounded symmetric domains arising from linear sections of minimal embeddings of their compact duals}, Math. Z. (2017) {\bf 286}: 679--700; DOI 10.1007/s00209-016-1778-7.

\bibitem[CXY17]{CXY17} Chan, S. T., Xiao, M., Yuan, Y.: \emph{Holomorphic isometries between products of complex unit balls}, Internat. J. Math. {\bf 28} (2017), 1740010, 22 pp.

\bibitem[Do91]{Do91} Dor, A.: \emph{Holomorphic maps from the unit ball of $\mathbb C^N$ that take hyperplanes into hyperplanes}, Math. Ann. {\bf 289}, 473--490 (1991)

\bibitem[Do96]{Do96} Dor, A.: \emph{A domain in $\mathbb C^m$ not containing any proper image of the unit disc}, Math. Z. {\bf 222}, pp. 615--625 (1996)

\bibitem[Dr02]{Dr02} Drinovec Drnov\v{s}ek, B.: \emph{Discs in Stein manifolds containing given discrete sets}, Math. Z. {\bf 239}, pp. 683--702 (2002)

\bibitem[For03]{For03}
Forstneri\v{c}, F.: \emph{Noncritical holomorphic functions on Stein manifolds}, Acta Math. {\bf 191} (2003), 143--189.

\bibitem[For17]{For17}
Forstneri\v{c}, F.: \emph{Stein manifolds and holomorphic mappings. The homotopy principle in complex analysis}, (Second edition), volume 56 of Ergebnisse der Mathematikund ihrer Grenzgebiete. 3. Folge. A Series of Modern Surveys in Mathematics.
Springer, Berlin, 2017.

\bibitem[Mok89]{Mok89} Mok, N.: \emph{Metric rigidity theorems on Hermitian locally symmetric manifolds}, Series in Pure Mathematics, Vol. 6, World Scientific Publishing Co., Singapore; Teaneck, NJ, 1989.


\bibitem[Mok12]{Mok12} Mok, N.: \emph{Extension of germs of holomorphic isometries up to normalizing constants with respect to the Bergman metric}, 
J. Eur. Math. Soc. (JEMS) {\bf 14} (2012),
1617--1656.

\bibitem[Mok16]{Mok16} Mok, N.: \emph{Holomorphic isometries of the complex unit ball into irreducible bounded symmetric domains}, Proc. Amer. Math. Soc. {\bf 14}4 (2016), 
pp. 4515--4525.

\bibitem[Mok22]{Mok22} Mok, N.: \emph{Holomorphic retractions of bounded symmetric domains onto totally geodesic complex submanifolds}, Chinese Ann. Math. Ser. B{\bf 43} (2022), 1125--1142.

\bibitem[MN17]{MN17}
Mok, N., Ng, S.-C.: \emph{On compact splitting complex submanifolds of quotients of bounded symmetric domains}, Sci. China Math. {\bf 60} (2017), 1057--1076.

\bibitem[Ru80]{Rudin80} Rudin, W.: \emph{Function theory in the unit ball of $\mathbb C^n$}, Grundlehren der Mathematischen Wissenschaften {\bf 241}, Springer, New York, 1980.


\bibitem[Wo72]{Wo72} Wolf, J. A.: \emph{Fine structure of Hermitian symmetric spaces}, Symmetric spaces (Short Courses, Washington Univ., St. Louis, Mo., 1969--1970), pp. 271--357,
Pure Appl. Math., Vol. 8
Marcel Dekker, Inc., New York, 1972.
\end{thebibliography}
\end{document}